\documentclass[12pt,draft]{article}

\usepackage{mathtools}

\usepackage{amssymb, amsmath}
\usepackage{amsthm}
\usepackage{xcolor}
\usepackage{tikz}
\usepackage{url}
\usepackage{fullpage}

\DeclareSymbolFont{bbold}{U}{bbold}{m}{n}
\DeclareSymbolFontAlphabet{\mathbbold}{bbold}

\newcommand{\N}{\mathbb{N}}
\newcommand{\Z}{\mathbb{Z}}

\newcommand{\R}{\mathbb{R}}
\newcommand{\C}{\mathbb{C}}

\renewcommand{\epsilon}{\varepsilon}

\renewcommand{\i}{{\rm i}}



\DeclareMathOperator{\grad}{grad}
\DeclareMathOperator{\curl}{curl}

\DeclareMathOperator{\dive}{div}
\DeclareMathOperator{\rge}{ran}
\DeclareMathOperator{\kar}{ker}
\DeclareMathOperator{\dom}{dom}

\DeclareMathOperator{\m}{m}
\renewcommand{\Re}{\operatorname{Re}}

\DeclareMathAccent{\Circ}{\mathalpha}{operators}{"17}

\let\eps\varepsilon
\let\phi\varphi
\let\le\leqslant
\let\leq\leqslant
\let\ge\geqslant
\let\geq\geqslant

\makeatletter
\def\@row#1,{#1\@ifnextchar;{\@gobble}{&\@row}}
\def\@matrix{%
	\expandafter\@row\my@arg,;%
	\@ifnextchar({\\ \get@in@paren{\@matrix}}{\after@matrix}%
}
\def\matrixtype#1#2#3{%
	\ifmmode\def\after@matrix{\end{#2}\right#3}%
\else\def\after@matrix{\end{#2}\right#3$}$\fi
\left#1\begin{#2}\get@in@paren{\@matrix}$%
}
\def\@column#1,{#1\@ifnextchar;{\@gobble}{\\ \@column}}
\newcommand\vect{}
\def\svect(#1){\left(\begin{smallmatrix}\@column#1,;\end{smallmatrix}\right)}
\def\vect{\get@in@paren{\@vect}}
\def\@vect{\left(\begin{matrix}\expandafter\@column\my@arg,;\end{matrix}\right)}
\def\get@in@paren#1({\def\my@arg{}\def\my@rest{}\def\after@get{#1}\get@arg}
\let\e@a\expandafter
\def\get@arg#1){\e@a\kl@test\my@rest#1(;}
\def\kl@test#1(#2;{\e@a\def\e@a\my@arg\e@a{\my@arg#1}%
	\ifx:#2:\let\my@exec\after@get
	\else\let\my@exec\get@arg
	\e@a\def\e@a\my@arg\e@a{\my@arg(}%
	\def@rest#2;%
	\fi\my@exec}
\def\def@rest#1(;{\def\my@rest{#1\kl@zu}}
\def\kl@zu{)}

\makeatletter
\newcommand\MyPairedDelimiter{%
	\@ifstar{\My@Paired@Delimiter{{}}}
	{\My@Paired@Delimiter{}}%
}
\newcommand\My@Paired@Delimiter[4]{%
	\newcommand#2{%
		\@ifstar{\start@PD{#1}{\delimitershortfall=-1sp}{#3}{#4}}
		{\start@PD{#1}{}{#3}{#4}}%
	}%
}
\newcommand\start@PD[5]{%
	#1\mathopen{\mathpalette\put@delim@helper{\put@delim{#2}{#3}{.}{#5}}}%
	#5%
	\mathclose{\mathpalette\put@delim@helper{\put@delim{#2}{.}{#4}{#5}}}%
}
\newcommand\put@delim@helper[2]{%
	\hbox{$\m@th\nulldelimiterspace=0pt #2#1$}%
}
\newcommand\put@delim[5]{%
	\setbox\z@\hbox{$\m@th#5{#4}$}%
	\setbox\tw@\null
	\ht\tw@\ht\z@ \dp\tw@\dp\z@
	#1#5%
	\left#2\box\tw@\right#3%
}

\makeatother
\MyPairedDelimiter*{\abs}{\lvert}{\rvert}
\MyPairedDelimiter*{\norm}{\lVert}{\rVert}
\MyPairedDelimiter{\set}{\{}{\}}

\theoremstyle{plain} 
\newtheorem{theorem}{Theorem}[section]
\newtheorem{corollary}[theorem]{Corollary}
\newtheorem{lemma}[theorem]{Lemma}
\newtheorem{proposition}[theorem]{Proposition}
\theoremstyle{definition}
\newtheorem{example}[theorem]{Example}

\newtheorem{hypothesis}[theorem]{Hypothesis}
\newtheorem{definition}[theorem]{Definition}
\newtheorem{remark}[theorem]{Remark}

\usepackage{enumitem}

\newcommand{\ep}{\varepsilon}
\setenumerate[1]{nolistsep} 
\setenumerate[2]{nolistsep} 

\begin{document}

\medmuskip=4mu plus 2mu minus 3mu
\thickmuskip=5mu plus 3mu minus 1mu
\belowdisplayshortskip=9pt plus 3pt minus 5pt

\title{Fibre Homogenisation}

\author{Shane Cooper and Marcus Waurick}

\date{}

\maketitle

\begin{abstract}
In this article we present a novel method for studying the asymptotic behaviour, with order-sharp error estimates,  of the resolvents of parameter-dependent operator families. The method is  applied to the study of differential equations with rapidly oscillating coefficients in the context of  second-order PDE systems and the Maxwell system. This produces a non-standard homogenisation result that is characterised by `fibre-wise' homogenisation of the related Floquet-Bloch PDEs. These fibre-homogenised resolvents are shown to be asymptotically equivalent to a whole class of operator families, including those obtained by standard homogenisation methods.
\end{abstract}

\section{Introduction}
%
\hspace{1em} This article is concerned with the asymptotic analysis of  parameter-dependent operators that admit a fibre decomposition. Such families appear for example in the asymptotic analysis of differential operators with rapidly oscillating periodic coefficients $\mathcal{B}_\ep$ defined in the whole space $L^2(\mathbb{R}^d)$. In this example context, the   period of the coefficients is the parameter $\ep$ and a typical goal is to understand the behaviour of solutions $u_\ep$, for a given force $f$, to
\[
\mathcal{B}_\ep u_\ep = f
\]
for small $\ep$.

\noindent A well-known approach  to determine the asymptotic behaviour of $u_\ep$ is the process of homogenisation (for which there is a vast body of literature available, see for example \cite{BeLiPa}, \cite{JKO} for an introduction to the field). In this process, the sequence $u_\ep$ is typically determined to converge, in an appropriate sense, to a limit $u$ and then one aims to establish the existence of an `homogenised' operator for which the identity $u = \mathcal{B}^{-1} f$ holds.  Upon establishing the homogenised operator, one can subsequently ask about the magnitude, in an appropriate metric, of the  difference $u_\ep - u = (\mathcal{B}_\ep^{-1} - \mathcal{B}^{-1} ) f $. Quantifying this error, uniformly in $\ep$ and $f$, is important, for example, in determining the asymptotic behaviour of the spectral properties of the family $\mathcal{B}_\ep$ and in the study of evolution problems $(\tfrac{d}{dt})^\alpha u_\ep + \mathcal{B}_\ep u_\ep = f$, $\alpha \in\{1,2\}$. 

\noindent In the context of second-order differential periodic operators, error estimates of the order $\sqrt{\ep}$ have been known for some time, see for example \cite{JKO}. While, the expected (order-sharp) order $\ep$ error estimates for $L^2(\mathbb{R}^d)$ right-hand side where first obtained in the works of Birman-Suslina \cite{BiSu}. Therein, they utilise the fact that $L^2(\mathbb{R}^d)$ is unitarily equivalent, via the Gelfand transform, to the space $L^2( [-\pi,\pi)^d; L^2( [0,1)^d))$, and that the operator $\mathcal{B}_\ep$ is unitarily equivalent to the fibre integral $\int_\Theta^\oplus \mathcal{B}_\ep(\theta) d\theta$ where $\mathcal{B}_\ep(\theta)$ is the second-order differential operator accompanied with quasi-periodic boundary conditions. Their subsequent analysis then focuses on this decomposition and a spectral study of the resolvents of $\mathcal{B}_\ep$ in a neighbourhood of the bottom of the spectrum. The idea of a spectral study via the Gelfand transform had been used previously in the works \cite{CoVa,J1} to obtain error estimates in homogenisation; although these works did not obtain order-sharp estimates in the uniform-operator topology. Very recently, in \cite{Se} the homogenisation with order-sharp operator-norm error estimates is established for second-order periodic operators with non-selfadjoint coefficients that admit global slowly varying and local rapidly oscillating dependence. 
 We mention for completeness, that in context of  second-order elliptic systems with periodic coefficients in bounded domains, error estimates in homogenisation of the order $\ep |\ln \ep|^\alpha$, $\alpha >0$, have been obtained by different techniques in the works \cite{KeLiSh,JPa}; order-sharp estimates were obtained  in bounded domains: for scalar equations using periodic unfolding in \cite{Gr}, and for systems, using combinations of the techniques in  \cite{BiSu} and \cite{JPa},  in \cite{Su1,Su2}.

{ 
On the subject of evolution(ary) problems, we make comments relevant to this article on the works \cite{Wa1,Wa2,Wa3}. In these works,  the homogenised systems for various time-dependent problems posed in bounded domains are obtained  by an interesting projection based technique. This projection technique was recently combined with the Gelfand transform to provide order-sharp error estimates between resolvents of the full time-dependent one-dimensional visco-elastic operator and its homogenised limit, see \cite{ChWa}. Therein, the method of proof relied on the one-dimensional nature of the problem and the so-called Schur complement.
}

In this article, our main focus of study is the behaviour of resolvents of parameter-dependent families of fibre-integral operators $\int_\Theta^\oplus \mathcal{B}_\ep(\theta) d\theta $ on a space $\int_{\Theta}^{\oplus} H d\theta$,
where
\[
\mathcal{B}_\ep(\theta) = M(\theta) + \tfrac{1}{\ep} A(\theta),
\]
for bounded linear $M(\theta)$ and possibly unbounded linear skew-selfadjoint $A(\theta)$. We are interested in studying the behaviour of $\mathcal{B}_\ep(\theta)^{-1}$ in the uniform-operator topology, uniform in $\theta$, for small $\ep$. Unlike in standard homogenisation approaches, where one would determine a so-called homogenised limit operator $\mathcal{B}$ for a given $\mathcal{B}_\ep$ and then determine bounds on the difference $\mathcal{B}_\ep^{-1} - \mathcal{B}^{-1}$ (via the fibre-integral representation or otherwise),  we emphasise here that we directly analyse the behaviour of $\mathcal{B}_\ep(\theta)^{-1}$  for sufficiently small, non-zero, $\ep$. The reason we adopt this approach is that, in general, the  point-wise (in $\theta$) homogenised limits (in $\ep$) of the operators $\mathcal{B}_\ep(\theta)$ are not the uniform limits. As such, to obtain error estimates one would need to come up with an approach to reconcile this difference and produce uniform in $\theta$ error bounds. (We mention in passing that in the context of high-contrast homogenisation of second-order differential operators, order-sharp operator-norm error estimates where obtained, in \cite{ChCo}, upon the recovery of uniform limits from point-wise limits  by an operator-theoretic analogue of matched asymptotic expansions.) Here, we develop a new method of studying the uniform in fibre behaviour of resolvents to fibre-integral families in terms of the small parameter. This method, exposed in Section \ref{s:gen}, is based on the observation that the lack of uniformity of the point-wise asymptotics of $\mathcal{B}_\ep(\theta)$ is due to the fact that spectrum of the operator family $(A(\theta))_\theta$ intersects zero for certain values of $\theta$. Therefore, to study the asymptotics, our method revolves around decomposing the underlying Hilbert space $H$ into a space $R(\theta)$ in which this operator $A(\theta)$ is uniformly invertible and its orthogonal complement $N(\theta)$. Subsequently, we can decompose the operator $\mathcal{B}_\ep(\theta)$ into uniformly invertible and singular parts; { this decomposition is based on developing the projection technique used in \cite{Wa1,Wa2,Wa3} and \cite{ChWa}. (We comment though that our approach does not need to rely on existence of the inverse to the Schur-complement. This improves the constants-of-error obtained in the uniform-operator norm bounds.)} Upon such a decomposition, it is a simple task to then determine that the uniform leading-order behaviour, for small $\ep$, of the family $\mathcal{B}_\ep(\theta)$ in the uniform-operator topology  is given by the projection of $\mathcal{B}_\ep(\theta)$ to $N(\theta)$, see Theorem \ref{t:mtgr} and Proposition \ref{p:sc3}. Remarkably, and the reason why we coin this method fibre homogenisation, is that this projection in the context of differential operators with rapidly oscillating coefficients gives rise to a fibre-dependent analogue to the standard homogenised coefficients, from classical theory, that is asymptotically equivalent to but, in general, different to the traditional homogenised matrix. This is the subject of Sections \ref{s:example} and \ref{sec:ahom}. Additionally, as a bi-product of this analysis we determine a whole family of operators that are asymptotically equivalent (in terms of resolvents) to the operator $\mathcal{B}_\ep$; these operators are characterised by being equal to $\mathcal{B}_{\ep}(\theta)$ on the space $N(\theta)$; this statement is made precise in  Theorem \ref{t:hom2}.

In closing, a consequence of the analysis in this article is that we present new results which capture the leading-order singular behaviour, in operator-norm, of the resolvents of fibre-integral operator families depending on a small parameter. These results in turn allow one to describe a whole class of asymptotically equivalent operator families, including those found by standard homogenisation methods (in the context of differential operators with rapidly oscillating coefficients).   The method presented in this article is not confined to the study of self-adjoint operator families arriving from second-order PDE systems; the scheme admits for example second-order PDE systems with non-selfadjoint coefficients, see Section \ref{s:example} as well as the Maxwell system, see Section \ref{s:max}. Moreover, our study easily fits into the static variants of the framework of evolutionary equations developed by Picard et al., see, e.g., \cite[Chapter 6]{Pi} or \cite{PiTrWaWe}. {In particular, we provide quantitative estimates for the first time to static variants of the systems in \cite{Wa1,Wa2,Wa3}.}

\section{Abstract fibre homogenisation}\label{s:gr}
\label{s:gen}
Let $\Theta$ be a non-empty set. For a given family of Hilbert spaces $(H_\theta)_{\theta\in \Theta}$, $\ep \in (0,\infty)$,  $M(\theta)\in L(H_\theta)$ with $\|M\|_\infty\coloneqq \sup_{\theta\in\Theta} \|M(\theta)\|<\infty$, and $A(\theta)\colon \dom(A(\theta))\subseteq H_\theta \to H_\theta$ densely defined and closed, we consider the operator family
\[
\mathcal{B}_\ep(\theta) : = M(\theta)+\tfrac{1}{\eps}A(\theta).
\]
Under the assumptions that there exists a $c \in (0,\infty)$ such that
\begin{equation}
\label{assump1}
\begin{aligned}
 \forall\, \theta \in \Theta: & \quad & \Re M(\theta) : = \tfrac{1}{2} (M(\theta) + M(\theta)^*)\geq c,  & \quad \text{and} \quad & A(\theta)=-A(\theta)^*,
\end{aligned}
\end{equation}
the operator $\mathcal{B}_\ep(\theta)$ is invertible for all $\ep, \theta$, cf.\ Lemma \ref{l:invblty} below. Typically, in homogenisation problems, fibre integral operators of the form $\int^{\oplus}_{\Theta} \mathcal{B}_\ep(\theta) d\theta$ appear. For example via the Gelfand transform for differential operators with periodic coefficients, see Sections \ref{s:example} and \ref{s:max}. A means to address the asymptotics, as $\ep$ tends to zero,  of such operators is to consider the behaviour of the resolvents for small $\ep$.
For this reason, we are interested in studying the uniform in $\theta$ behaviour for small $\ep$  for the inverse operators $\mathcal{B}_\ep(\theta)^{-1}$.

We now provide a general set of assumptions that, if satisfied, allow one to construct such asymptotics.
\begin{hypothesis}\label{hyp:gr} Assume for all $\theta \in \Theta$, there exists a closed subspace $N(\theta) \subseteq H_\theta$ with $R(\theta)\coloneqq N(\theta)^\perp$ such that, for the canonical embeddings $\iota_{N(\theta)}\colon N(\theta)\hookrightarrow H_\theta$, $\iota_{R(\theta)}\colon R(\theta)\hookrightarrow H_\theta$ and the orthogonal projections $\pi_{N(\theta)}\coloneqq \iota_{N(\theta)}\iota_{N(\theta)}^*$  $\pi_{R(\theta)}\coloneqq \iota_{R(\theta)}\iota_{R(\theta)}^*$, the following conditions hold: \\
  \begin{enumerate}[label=(\alph*)]
  	    \item\label{gr1} $A(\theta)\pi_{N(\theta)}$ is bounded for all $\theta\in \Theta$.
\item\label{gr2}{$\pi_{R(\theta)} A(\theta)\subseteq A(\theta)\pi_{R(\theta)}$ for all $\theta \in \Theta$.}
   \item\label{gr3}  $\iota_{R(\theta)}^*A(\theta)\iota_{R(\theta)}$ is, uniformly in $\theta$, boundedly invertible:
   \begin{equation}\label{eq:hypCR}
      C_R\coloneqq \sup_{\theta\in \Theta} \|\left(\iota_{R(\theta)}^*A(\theta)\iota_{R(\theta)}\right)^{-1}\|_{L(R(\theta))}<\infty.
   \end{equation}
  \end{enumerate}
\end{hypothesis}
 The main theorem of this section is as follows.
\begin{theorem}\label{t:mtgr} Assume \eqref{assump1} and Hypothesis \ref{hyp:gr}. Then, for all $\epsilon\in\big(0, 1 /2 C_R \Vert M \Vert_\infty\big)$, $\theta \in \Theta$ one has 
\[
 \| \mathcal{B}_\epsilon(\theta)^{-1} - \big(\pi_{N(\theta)}M(\theta)\pi_{N(\theta)}+\tfrac{1}{\epsilon}A(\theta)\big)^{-1} \| \\ \leq  \kappa(\|M\|_\infty,C_R,c) \eps,
\]
where
\[
  \kappa(\|M\|_\infty,C_R,c)\coloneqq  2C_R \big( 1 + \tfrac{\Vert M \Vert_\infty}{c} \big)^2 + C_R.
\]
\end{theorem}
\begin{remark}{ \ } 
	
	\begin{enumerate}
		\item{The existence of $\big(\pi_{N(\theta)}M(\theta)\pi_{N(\theta)}+\tfrac{1}{\epsilon}A(\theta)\big)^{-1}$ is addressed in the proof of Theorem \ref{t:mtgr}.}
		\item{ For convenience of the reader and to keep the statements that follow as accessible as possible, we do not record the explicit number $\kappa(\|M\|_\infty,C_R,c)$ in front of $\epsilon$ and just write $\kappa$. We emphasise, however, the following asymptotic properties:
			\begin{align*}
			&\limsup_{c\to 0} c^2\kappa(\|M\|_\infty,C_R,c) = 2 C_R\Vert M \Vert_\infty <\infty,
			\\ & \limsup_{C_R\to \infty} \tfrac{\kappa(\|M\|_\infty,C_R,c)}{C_R} =  2 \big( 1 + \tfrac{\Vert M \Vert_\infty}{c} \big)^2 + 1<\infty,\text{ and}
			\\ & \limsup_{\|M\|_\infty\to \infty} \tfrac{\kappa(\|M\|_\infty,C_R,c)}{\|M\|_\infty^2} =  \tfrac{2C_R}{c^2}  <\infty.
			\end{align*} 
			Most prominently, the last equality becomes important, if one wants to study time-dependent problems, see \cite{ChWa}. The decisive observation frequently used in the present text is that $\kappa(\|M\|_\infty,C_R,c)$ is \emph{independent of} $\epsilon>0$ (if sufficiently small) \emph{and all} $\theta\in\Theta$.
		}
		\item{ We remark here that $\|\mathcal{B}_\ep(\theta)^{-1}\|_{L(H_\theta)}\leq 1/c$, see Corollary \ref{c:binv} below. Moreover, it is possible to show that $\|\big(\pi_{N(\theta)}M(\theta)\pi_{N(\theta)}+\tfrac{1}{\epsilon}A(\theta)\big)^{-1}\|\leq \max\{\frac{1}{c},\epsilon C_R\}$ for all $\ep>0$ and $\theta\in \Theta$, also see Proposition \ref{p:sc1}. Hence, it is possible to prove an estimate of the form
		\[
		   \|\mathcal{B}_\epsilon(\theta)^{-1} - \big(\pi_{N(\theta)}M(\theta)\pi_{N(\theta)}+\tfrac{1}{\epsilon}A(\theta)\big)^{-1} \| \leq \tilde\kappa(\|M\|_\infty,C_R,c)\epsilon
		\] with $\tilde\kappa$ satisfying a similar asymptotic behavior as $\kappa$: 
                \[
                  \limsup_{c\to 0} c^2\tilde\kappa(\|M\|_\infty,C_R,c),\, \limsup_{C_R\to \infty} \tfrac{\tilde\kappa(\|M\|_\infty,C_R,c)}{C_R},\,  \limsup_{\|M\|_\infty\to \infty} \tfrac{\tilde\kappa(\|M\|_\infty,C_R,c)}{\|M\|_\infty^2}<\infty.			
                \]For this reason, we may also drop the condition that $\ep$ has to be sufficiently small. We choose to do this for the remainder of the manuscript.
		}
	\end{enumerate}
\end{remark}
Theorem \ref{t:mtgr} does not only provide us with leading-order asymptotics of $\mathcal{B}_\ep^{-1}$, it presents a way of comparing two operator families that `coincide' on $N(\theta)$. 
 More precisely, the following result holds.
\begin{theorem}
\label{t:hom2}
Assume \eqref{assump1} and Hypothesis \ref{hyp:gr}. Consider, for $\theta \in \Theta$, $\widetilde{M}(\theta)  \in L(H_\theta)$  such that $\Vert \widetilde{M} \Vert_\infty < \infty$, with  $\forall \theta \in \Theta: \Re{\widetilde{M}(\theta)} \ge c$. Furthermore, assume that
\[
\pi_{N(\theta)}M(\theta)\pi_{N(\theta)} = \pi_{N(\theta)}\widetilde{M}(\theta)\pi_{N(\theta)} \quad (\theta \in \Theta).
\]
Then, there exists $\kappa >0$ such that for all $\theta \in \Theta$ and $\ep >0$ one has
\[
\Vert {\mathcal{B}}_\ep(\theta)^{-1} - \big(\widetilde{M}(\theta) + \tfrac{1}{\ep} A(\theta)\big)^{-1} \Vert \le \kappa \ep.
\]
\end{theorem}
\begin{proof}
The operator $\widetilde{\mathcal{B}}_\ep(\theta) \coloneqq \widetilde{M}(\theta) + \tfrac{1}{\ep} A(\theta)$ satisfies the assumptions of Theorem \ref{t:mtgr} and then the desired result follows from the triangle inequality and the fact 
\[
\pi_{N(\theta)}(M(\theta)-\widetilde{M}(\theta))\pi_{N(\theta)} = 0 \quad (\theta \in \Theta).\qedhere
\]
\end{proof}
The remainder of this section will be dedicated to the proof of  Theorem \ref{t:mtgr}. We begin with providing a series of relevant preliminary results.

\begin{lemma}\label{l:invblty} Let $H$ be a Hilbert space, $M\in L(H)$ and $A\colon \dom(A)\subseteq H\to H$ be skew-selfadjoint. Assume that there exists $c>0$ such that $\Re M\geq c$. Then, the operator $M+A$ is continuously invertible and the inequality\[
   \|(M+A)^{-1}\|\leq \tfrac{1}{c}
\]
holds.
\end{lemma}
\begin{proof}
The observation that $\Re(M+A)=\Re(M+A)^*= \Re{M}\geq c$ on $\dom(M+A)=\dom(A)=\dom(A^*)=\dom((M+A)^*)$ implies, via a simple application of the Cauchy-Schwarz inequality, that the range of $M+A$ is closed, $M+A$ is boundedly invertible on its range and the kernel of $(M+A)^*$ is trivial. Then, we conclude the assertion from the orthogonal decomposition $H=\overline{\rge(M+A)}\oplus \kar(M+A)^*$.
\end{proof}

\begin{corollary}\label{c:binv} Under the assumptions \eqref{assump1},  $\mathcal{B}_\epsilon(\theta)$ is boundedly invertible and the inequality
\[
\sup_{\theta \in \Theta}   \|\mathcal{B}_\epsilon(\theta)^{-1}\|\leq \tfrac{1}{c}
\]
holds. 
\end{corollary}

\begin{lemma}\label{l:Athetainv} For a given Hilbert space $H$ and  $A\colon \dom(A)\subseteq H\to H$ densely defined, assume that there exists a closed subspace $U\subseteq H$ such that $\pi_U A\subseteq A\pi_U$, where  $\pi_U\colon H\to H$ is the orthogonal projection on $U$. Then, for $\pi_V\coloneqq (1-\pi_U)$ we obtain $\pi_V A\subseteq A\pi_V$ and 
\[
   \overline{\pi_V A\pi_U}=\overline{\pi_U A\pi_V}=0.
\]
\end{lemma}
\begin{proof}
 We compute $\pi_V A=(1-\pi_U)A=A-A\pi_U\subseteq A(1-\pi_U)=A\pi_V$. Hence, we obtain
 \[
    \pi_V A\pi_U\subseteq A\pi_V\pi_U=0\text{ and }\pi_U A\pi_V\subseteq A\pi_U\pi_V=0.
 \]
 The assertion now follows from the fact that both $\pi_V A\pi_U$ and $\pi_U A\pi_V$ are densely defined; indeed, the respective domains contain the domain of $A$. 
\end{proof}

\begin{lemma}\label{l:Asksa} Let $H$ be a Hilbert space and $A\colon \dom(A)\subseteq H\to H$ skew-selfadjoint. Assume that there exists $U\subseteq H$ closed such that  $\pi_U A \subseteq A \pi_U $ and $ A\pi_V$ bounded, where $\pi_U \colon H \to H$ denotes the orthogonal projection to $U$ and $\pi_V\coloneqq (1-\pi_U)$. Then $\iota_U^* A\iota_U$ and $\iota_V^* A\iota_V$ are skew-selfadjoint in $U$ and $V\coloneqq U^\bot$, respectively, where $\iota_U\colon U \hookrightarrow H$, $\iota_V\colon V \hookrightarrow H$. 
\end{lemma}
\begin{proof}
 First of all, note that the assertion that $\iota_U^* A\iota_U$ (resp. $\iota_V^* A\iota_V$) is skew-selfadjoint is equivalent to  $\pi_UA\pi_U$ (resp. $\pi_V A\pi_V$) being skew-selfadjoint.
 
  It is easy to see that $\pi_UA\pi_U$ is skew-Hermitian. Moreover, the inclusion $\pi_U^2 A \subseteq \pi_U A\pi_U$ implies that  $\pi_UA\pi_U$ is densely defined  and, thus, skew-symmetric.

  By Lemma \ref{l:Athetainv}, the same reasoning applies to $\pi_V A\pi_V$. Thus, as $A\pi_V$ is bounded we deduce that $\pi_VA\pi_V$ is skew-selfadjoint. 
   
   We  now prove that $\pi_U A \pi_U$ is skew-selfadjoint. Note that  $\phi \in \dom(A)$ if, and only if, $\pi_U\phi\in \dom(A)$. Indeed, the necessary implication follows from $\pi_U A\subseteq A\pi_U$; sufficiency follows from  $A\pi_V$ being bounded which, in turn, implies that $\pi_V\psi\in \dom(A)$ for all $\psi\in H$. Therefore, we infer that $A=A\pi_U+A\pi_V$, and consequently, upon utilising Lemma \ref{l:Athetainv}, we calculate
\[
A =(\pi_U + \pi_V)A(\pi_U+\pi_V)  = \pi_U A\pi_U + \pi_V A \pi_V.\]
Finally, since $A$ and $\pi_V A\pi_V$ are skew-selfadjoint, and $\pi_V A\pi_V$ is bounded, it follows that $\pi_U A \pi_U$ is skew-selfadjoint.
\end{proof}

We now aim to provide a formula for $\mathcal{B}_\ep(\theta)^{-1}$, in terms of the space $N(\theta)$ and $R(\theta) = N(\theta)^\perp$, that will be utilised in the proof of Theorem \ref{t:mtgr}. First, some a priori observations.
\begin{proposition}
\label{p:sc1}
Assume \eqref{assump1}, Hypothesis \ref{hyp:gr} and recall $C_R$ from \eqref{eq:hypCR}. Let $\mathcal{B}_{\ep,N}(\theta) \in L(N(\theta))$, $\mathcal{B}_{\ep,R}(\theta) \in L(R(\theta))$ be given by
\[
\begin{aligned}
\mathcal{B}_{\ep,N}(\theta) &= \iota^*_{N(\theta)} M(\theta) \iota_{N(\theta)} + \tfrac{1}{\ep} \iota^*_{N(\theta)} A \iota_{N(\theta)}, \quad \text{and}\\
 \mathcal{B}_{\ep,R}(\theta) &= \iota^*_{R(\theta)} M(\theta) \iota_{R(\theta)} + \tfrac{1}{\ep} \iota^*_{R(\theta)} A \iota_{R(\theta)}.
\end{aligned}
\]
 Then, the following assertions hold.
 \begin{enumerate}[label=(\alph*)]
 	\item\label{en:lsgr1} Let $\epsilon_0\coloneqq 1/(2C_R\|M\|_\infty)$. Then, for all $\epsilon\in(0,\epsilon_0)$ and $\theta \in \Theta$, the operator $\mathcal{B}_{\epsilon,R}(\theta)$ is continuously invertible and
 	\[
 	\sup_{\theta\in \Theta} \| \mathcal{B}_{\epsilon,R}(\theta)^{-1}\|\leq 2C_R \ep.
 	\]
 	\item\label{en:lsgr1.5} For all $\epsilon>0$ and $\theta\in \Theta$, the operator $\mathcal{B}_{\epsilon,N}(\theta)$ is continuously invertible, and
 	\[
 	\sup_{\theta\in \Theta}\left\| \mathcal{B}_{\epsilon,N}(\theta)^{-1}\right\|\leq \tfrac{1}{c}.
 	\]
 \end{enumerate}
\end{proposition}
 \begin{proof} For \ref{en:lsgr1}, we proceed as follows. By Hypothesis \ref{hyp:gr}, the operator $A_R(\theta) \coloneqq \iota^*_{R(\theta)} A(\theta) \iota_{R(\theta)}$ is continuously invertible. Hence, we obtain
 \[
   \mathcal{B}_{\epsilon,R}(\theta)=\tfrac{1}{\ep} A_R(\theta)\big( \ep A_R(\theta)^{-1}\iota^*_{R(\theta)}M(\theta)\iota_{R(\theta)}+1\big).
 \]
 From the inequality
 \[
    \|\epsilon A_R(\theta)^{-1}\iota^*_{R(\theta)}M(\theta)\iota_{R(\theta)}\|\leq \epsilon C_R \|M\|_\infty\leq \tfrac{1}{2},
 \]
 we deduce via a Neumann series argument, for the inverse of $1 + \epsilon A_R(\theta)^{-1}\iota^*_{R(\theta)}M(\theta)\iota_{R(\theta)}$, that
 \[
\mathcal{B}_{\epsilon,R}(\theta)^{-1}=\ep \sum_{k=0}^\infty \left(-\epsilon A_R(\theta)^{-1}\iota^*_{R(\theta)}M(\theta)\iota_{R(\theta)}\right)^kA_R(\theta)^{-1}.
 \]
 Thus, 
 \[
    \|\mathcal{B}_{\epsilon,R}(\theta)^{-1}\|\leq \ep C_R \sum_{k=0}^\infty \frac{1}{2^k}=2C_R \ep.
 \]
 For the proof of \ref{en:lsgr1.5}, we observe that, by Lemma \ref{l:Asksa}, the operator $A_{N(\theta)} \coloneqq \iota^*_{N(\theta)} A(\theta) \iota_{N(\theta)}$ is skew-selfadjoint. Hence, $\Re \mathcal{B}_{\epsilon,N}(\theta)\geq c$ and, thus, Lemma \ref{l:invblty} implies that $\mathcal{B}_{\epsilon,N}(\theta)^{-1}$ exists with $\|\mathcal{B}_{\epsilon,N}(\theta)^{-1}\|\leq 1/c$.
 \end{proof}
The following result holds.
\begin{proposition}
	\label{p:sc2}
Assume \eqref{assump1}, Hypothesis \ref{hyp:gr}  and let $\ep_0$ be as in Proposition \ref{p:sc1}. Then, for all $\ep \in (0,\ep_0)$ and $\theta \in \Theta$, the following assertions hold.
\begin{enumerate}[label=(\alph*)]
\item{ $	\pi_{N(\theta)}\mathcal{B}_{\epsilon}(\theta)^{-1}  =  \iota_{N(\theta)}\mathcal{B}_{\epsilon,N}(\theta)^{-1} \big( \iota^*_{N(\theta)} - \iota^*_{N(\theta)}M(\theta) \pi_{R(\theta)} \mathcal{B}_{\epsilon}(\theta)^{-1} \big);$}\label{psc2e1}
\item{ $\pi_{R(\theta)}\mathcal{B}_{\epsilon}(\theta)^{-1} = \iota_{R(\theta)}\mathcal{B}_{\epsilon,R}(\theta)^{-1} \big(\iota^*_{R(\theta)} - \iota^*_{R(\theta)}M(\theta)\pi_{N(\theta)}\mathcal{B}_{\epsilon}(\theta)^{-1}\big).  $}\label{psc2e2}
\end{enumerate}

\end{proposition}
\begin{proof}
Fix, $\ep, \theta$ and $f \in H_\theta$, and let $u = \mathcal{B}_{\epsilon}(\theta)^{-1} f$. Then $u = \iota_{N(\theta)}u_N + \iota_{R(\theta)}u _R$, where  $u_N = \iota^*_{N(\theta)} u$ and $u_R = \iota^*_{R(\theta)} u$. Now, by Lemma \ref{l:Athetainv}, one has 
\begin{align*}
\pi_{R(\theta)} A(\theta) = \pi_{R(\theta)} A(\theta) \pi_{R(\theta)}, & \quad &\pi_{N(\theta)} A(\theta) = \pi_{N(\theta)} A(\theta) \pi_{N(\theta)}.
\end{align*}
Consequently, with $A_R(\theta)=\iota_{R(\theta)}^*A(\theta)\iota_{R(\theta)}$ 
\begin{align*}
\pi_{R(\theta)} f & = \pi_{R(\theta)} \mathcal{B}_{\epsilon}(\theta) u \\ 
& =  \pi_{R(\theta)}  M(\theta) \pi_{N(\theta)} u + \pi_{R(\theta)}  M(\theta) \pi_{R(\theta)} u  +   \tfrac{1}{\ep} \iota_{R(\theta)} A_{R}(\theta) u_R\\
& =  \pi_{R(\theta)}  M(\theta) \pi_{N(\theta)} u + \iota_{R(\theta)}\mathcal{B}_{\epsilon,R}(\theta) u_R
\end{align*}
and, therefore,
\[
u_R =  \mathcal{B}_{\epsilon,R}(\theta)^{-1} \big( \iota^*_{R(\theta)} - \iota^*_{R(\theta)}  M(\theta) \pi_{N(\theta)}\mathcal{B}_{\epsilon}(\theta)^{-1}  \big) f.
\]
Similarly, we deduce that
\[
u_N = \mathcal{B}_{\epsilon,N}(\theta)^{-1} \big( \iota^*_{N(\theta)} - \iota^*_{N(\theta)} M(\theta) \pi_{R(\theta)} \mathcal{B}_{\epsilon}(\theta)^{-1}\big) f,
\]
and  the desired identities follow.
\end{proof}
We are now in the position to study the behaviour of  the inverse of $\mathcal{B}_\ep(\theta)$ for small $\ep$. 
\begin{proposition}\label{p:sc3} Assume \eqref{assump1}, Hypothesis \ref{hyp:gr} and let $\ep_0$ be as in Proposition \ref{p:sc1}. Then, for all $\ep \in (0,\ep_0)$ and $\theta \in \Theta$, the inequality
	\[
	\Vert \mathcal{B}_{\epsilon}(\theta)^{-1} -  \iota_{N(\theta)}\mathcal{B}_{\epsilon,N}(\theta)^{-1}  \iota^*_{N(\theta)}  \Vert \le 2C_R \big( 1 + \tfrac{\Vert M \Vert_\infty}{c} \big)^2 \ep
	\]
	holds. Here $C_R$ is given in Proposition \ref{p:sc1} \ref{en:lsgr1}.
\end{proposition}
\begin{proof} 
	The inequalities in Corollary \ref{c:binv} and Proposition \ref{p:sc1} \ref{en:lsgr1} imply that
		\[
		\sup_{\theta \in \Theta}\Vert \iota_{R(\theta)}\mathcal{B}_{\epsilon,R}(\theta)^{-1} \big( \iota^*_{R(\theta)} - \iota^*_{R(\theta)}  M(\theta) \pi_{N(\theta)}\mathcal{B}_{\epsilon}(\theta)^{-1}  \big) \Vert \le 2C_R \big( 1 + \tfrac{\Vert M \Vert_\infty}{c} \big) \ep.
		\]
By Proposition \ref{p:sc2} \ref{psc2e2},  Proposition \ref{p:sc1} \ref{en:lsgr1.5} and the  above assertion, we deduce that 
\[
		\sup_{\theta \in \Theta}\Vert \iota_{N(\theta)}  \mathcal{B}_{\epsilon,N}(\theta)^{-1}  \iota^*_{N(\theta)}M(\theta) \pi_{R(\theta)} \mathcal{B}_{\epsilon}(\theta)^{-1}  \Vert \le \tfrac{1}{c} \Vert M \Vert_\infty  2C_R \big( 1 + \tfrac{\Vert M \Vert_\infty}{c} \big) \ep.
\]
The proof of the proposition now follows from Proposition \ref{p:sc2} and the identity $\mathcal{B}_\ep(\theta)^{-1} = (\pi_{N(\theta)} + \pi_{R(\theta)}) \mathcal{B}_\ep(\theta)^{-1} $.  
\end{proof}
\begin{remark}
Proposition \ref{p:sc3} is one particular choice of the leading-order asymptotics for the inverse $\mathcal{B}_\ep(\theta)^{-1}$ and could be taken in the place of those presented in Theorem \ref{t:mtgr}. That being said, the reason we choose to demonstrate the equivalent asymptotics given by Theorem \ref{t:mtgr} is to present leading-order asymptotics for the resolvents of the operator $\mathcal{B}_\ep(\theta)$ that preserve $A(\theta)$.
\end{remark}

To complete the proof of Theorem \ref{t:mtgr} is now a simple task.
\begin{proof}[Proof of Theorem \ref{t:mtgr}]
To show that $\big( \pi_{N(\theta)} M(\theta) \pi_{N(\theta)} + \tfrac{1}{\ep} A(\theta) \big)^{-1}$ exists, observe that
\begin{equation}
\label{liminv}
\pi_{N(\theta)} M(\theta) \pi_{N(\theta)} + \tfrac{1}{\ep} A(\theta)  = \begin{pmatrix}
\iota_{N(\theta)} & \iota_{R(\theta)}
\end{pmatrix} \begin{pmatrix}
\mathcal{B}_{\ep,N}(\theta) & 0 \\ 0 & \tfrac{1}{\ep}A_{R}(\theta)	
\end{pmatrix} \begin{pmatrix}
\iota^*_{N(\theta)} \\ \iota^*_{R(\theta)}
\end{pmatrix},
\end{equation}
and that by Hypothesis \ref{hyp:gr}, $A_{R}(\theta)=\iota_{R(\theta)}^*A(\theta)\iota_{R(\theta)}$ is continuously invertible on $R(\theta)$ and by Proposition \ref{p:sc1} \ref{en:lsgr1.5},  $\mathcal{B}_{\ep,N}$ is continuously invertible on $N(\theta)$ for all  $\ep>0$ and $\theta \in \Theta$.

We compute with the help of \eqref{liminv}
\[
\begin{aligned}
\pi_{R(\theta)} \big( \pi_{N(\theta)} M(\theta) \pi_{N(\theta)} + \tfrac{1}{\ep} A(\theta) \big)^{-1} & = \ep \iota_{R(\theta)} A_{R}(\theta)^{-1}\iota^*_{R(\theta)}, \quad \text{and} \\
\pi_{N(\theta)} \big( \pi_{N(\theta)} M(\theta) \pi_{N(\theta)} + \tfrac{1}{\ep} A(\theta) \big)^{-1} & = \iota_{N(\theta)} \mathcal{B}_{\ep,N}(\theta)^{-1}\iota^*_{N(\theta)}.
\end{aligned}
\]
Then, the proof of the theorem follows by Hypothesis \ref{hyp:gr} \ref{gr3} and Proposition \ref{p:sc3}.
\end{proof}

\begin{remark}\label{r:br}
 Note that as an upshot of the method of proof, we observe that the leading-order asymptotics are in fact determined by the behaviour of the resolvent on the space $N(\theta)$ only, cf. Proposition \ref{p:sc3}. In particular, it is possible to replace $A_R(\theta)^{-1}$ by \emph{any} uniformly bounded linear operator acting in $R(\theta)$ in order to obtain an asymptotically equivalent answer to the assertion in Theorem \ref{t:mtgr}. In order to see this, one has to simply refer to \eqref{liminv}. In more formal terms, we have also proven the following result: Let $(T_\theta)_\theta$ be a family acting in $(L(H_\theta))_\theta$ be such that $\sup_{\theta \in \Theta }|| \iota^*_{R(\theta)}T(\theta)\iota_{R(\theta)} || < \infty$. Then, for all $\epsilon>0$ small enough and $\theta\in\Theta$ we have
\begin{multline*}
\bigg\|\mathcal{B}_\epsilon(\theta)^{-1}-\begin{pmatrix}\iota_{N(\theta)}&\iota_{R(\theta)}\end{pmatrix}\begin{pmatrix}
\mathcal{B}_{\epsilon,N}(\theta)^{-1}  & 0 \\ 0 & \ep  \iota_{R(\theta)}^*T(\theta)\iota_{R(\theta)} 
\end{pmatrix}\begin{pmatrix}\iota^*_{N(\theta)}\\\iota^*_{R(\theta)}\end{pmatrix}\bigg\| \\ \leq \Big( 2 C_R (1 + \tfrac{\Vert M \Vert_\infty}{c} )^2 + \sup_{\theta \in \Theta }|| \iota^*_{R(\theta)}T(\theta)\iota_{R(\theta)} || \Big)\epsilon.
\end{multline*}
\end{remark}

In applications it may happen that  $A(\theta)$ and $M(\theta)$ are realisations of a direct-fibre decomposition. Such a case presents no additional difficulty from the perspective of the above approach and one can argue in a similar manner as follows.

\begin{hypothesis}\label{hyp:grtheta}
  Let $H_0$ be a Hilbert space, $\Theta\subseteq \mathbb{R}^d$ measurable. For each $\theta\in \Theta$ let $H_\theta$ be a Hilbert space and assume there exists a Hilbert space $H$ such that $H_0=\int_{\Theta}^\oplus H$ and $H_\theta\subseteq H$ closed; set $\iota_\theta\colon H_\theta\hookrightarrow H$. For every $\theta\in \Theta$, let $M(\theta)\in  L(H_\theta)$, $A(\theta)\colon \dom(A(\theta))\subseteq H_\theta\to H_\theta$. We assume the following properties:
  \begin{enumerate}[label=(\alph*)]
   \item\label{en:grt1} for all $\theta\in \Theta$, $A(\theta)=-A(\theta)^*$,
   \item\label{en:grt2} $\Re M(\theta)\geq c$ for all $\theta\in \Theta$,
   \item\label{en:grt3} $A(\theta)$, $\theta\in\Theta$, satisfies Hypothesis \ref{hyp:gr},
    \item\label{en:grt5} assume that $\theta\mapsto \iota_\theta\big(M(\theta)+\frac{1}{\eps}A(\theta)\big)^{-1}\iota_\theta^*$ is weakly measurable.
  \end{enumerate}
  For $\epsilon>0$, consider
  \[
     \mathcal{C}_\epsilon \coloneqq  \int_{\Theta}^\oplus \iota_\theta\left(M(\theta)+\tfrac{1}{\eps}A(\theta)\right)^{-1}\iota_\theta^* d\theta.
  \]
\end{hypothesis}

\begin{theorem}\label{t:mtgrtheta}Assume Hypothesis \ref{hyp:grtheta}. Then,  there exists $\kappa>0$ such that for all $\epsilon>0$, the following inequality	\[
\left\| \mathcal{C}_\epsilon - \int_{\Theta}^\oplus \iota_\theta\left(\pi_{N(\theta)}M(\theta)\pi_{N(\theta)}+\tfrac{1}{\epsilon}A(\theta)\right)^{-1}\iota_\theta^* d\theta\right\|\leq \kappa \eps
\] holds.
\end{theorem}
\begin{proof}
 The proof follows from Theorem \ref{t:mtgr}. In fact, note that
 \begin{multline*}
   \mathcal{C}_\epsilon - \int_{\Theta}^\oplus \iota_\theta\left(\pi_{N(\theta)}M(\theta)\pi_{N(\theta)}+\tfrac{1}{\epsilon}A(\theta)\right)^{-1}\iota_\theta^* d\theta
   \\ = \int_{\Theta}^\oplus \iota_\theta\left(M(\theta)+\tfrac{1}{\eps}A(\theta)\right)^{-1}\iota_\theta^* d\theta - \int_{\Theta}^\oplus \iota_\theta\left(\pi_{N(\theta)}M(\theta)\pi_{N(\theta)}+\tfrac{1}{\epsilon}A(\theta)\right)^{-1}\iota_\theta^* d\theta
   \\ = \int_{\Theta}^\oplus \iota_\theta\left(\left(M(\theta)+\tfrac{1}{\eps}A(\theta)\right)^{-1}-\left(\pi_{N(\theta)}M(\theta)\pi_{N(\theta)}+\tfrac{1}{\epsilon}A(\theta)\right)^{-1}\right)\iota_\theta^* d\theta.
\end{multline*}
Thus, the asymptotic analysis requires estimating 
\[
\left(M(\theta)+\tfrac{1}{\eps}A(\theta)\right)^{-1}-\left(\pi_{N(\theta)}M(\theta)\pi_{N(\theta)}+\tfrac{1}{\epsilon}A(\theta)\right)^{-1}
\]
uniformly in $\theta$, which is done in Theorem \ref{t:mtgr}.
\end{proof}
The analogue of Proposition \ref{p:sc3} is as follows. 
\begin{theorem}\label{t:mtgrtheta2}Assume Hypothesis \ref{hyp:grtheta}. Then,  there exists $\kappa>0$ such that for all $\epsilon>0$, the following inequality	\[
	\left\| \mathcal{C}_\epsilon - \int_{\Theta}^\oplus \iota_\theta \iota_{N(\theta)}\big(\iota^*_{N(\theta)}\big( M(\theta)+\tfrac{1}{\epsilon}A(\theta) \big)\iota_{N(\theta)}\big)^{-1} \iota_{N(\theta)}^*\iota_\theta^* d\theta\right\|\leq \kappa\eps
	\] holds. 
\end{theorem}
\begin{proof}
Arguing as in the proof of Theorem \ref{t:mtgrtheta}, the asymptotic analysis requires estimating the difference 
\[
\left(M(\theta)+\tfrac{1}{\eps}A(\theta)\right)^{-1}-\iota_{N(\theta)}\big(\iota^*_{N(\theta)}\big( M(\theta)+\tfrac{1}{\epsilon}A(\theta) \big)\iota_{N(\theta)}\big)^{-1} \iota_{N(\theta)}^*,
\]
uniformly in $\theta$, which is given by Proposition \ref{p:sc3}.
\end{proof}
\section{Fibre homogenisation of second-order PDE systems with rapidly oscillating periodic coefficients}
\label{s:example}
In order to put the abstract result exposed in Section \ref{s:gen} into perspective, we shall study a classical example of homogenisation theory: an elliptic system of $n$ equations posed on $\mathbb{R}^d$ with rapidly varying periodic coefficients. For this, we denote $Y\coloneqq (0,1)^d$, and for a vector space $V$, denote $1_V\coloneqq (V\ni v\mapsto v\in V)$. For a subspace $V\subseteq L^2(Y)$ and functions $g,h\in L^2(Y)$, we denote
\begin{align*}
g\bot h&\hspace{0.05cm}\colon\hspace{-0.25cm}\Leftrightarrow \langle g,h\rangle_{L^2(Y)}=0, \qquad \text{ and } \qquad  V\bot h \coloneqq \{ v\in V; \langle v, h\rangle_{L^2(Y)} = 0\}.
\end{align*}
 We set 
\begin{align*}
   \mathcal{M}_{n,d}^\#& \coloneqq \big\{ a \in L^\infty(\mathbb{R}^d;L(\mathbb{C}^{n\times d})) \, ;  a(\cdot+k)=a(\cdot)\,(k\in \mathbb{Z}^d), \exists \nu>0: \Re a \ge \nu1_{\mathbb{C}^{n\times d}} \big\},   \\[5pt] 
   \mathcal{S}_{n,d}^\#&\coloneqq \big\{ s\in L^\infty(\mathbb{R}^d;L(\mathbb{C}^n)) \, ; \, s(\cdot+k)=s(\cdot)\,(k\in\mathbb{Z}^d),\, \exists \nu>0:\Re s\geq \nu1_{\mathbb{C}^n}\big\},
\end{align*}
\[
a_{ijkl}\coloneqq \langle a e_i\otimes e_j,e_k\otimes e_l\rangle_{\mathbb{C}^{n\times d}}\in L^\infty(\mathbb{R}^d)\quad( a\in \mathcal{M}_{n,d}^\#, \, i,k\in\{1,\ldots,n\},j,l\in\{1,\ldots,d\}),
\]
and
\[
s_{ij}\coloneqq \langle s e_i,e_j\rangle_{\mathbb{C}^n}\in L^\infty(\mathbb{R}^d)\quad (s\in \mathcal{S}_{n,d}^\#, \, i,j\in\{1,\ldots,n\}),
\]
where  $e_j$ is the $j$-th Euclidean basis vector.

For given $a\in \mathcal{M}_{n,d}^\#$, $s\in \mathcal{S}_{n,d}^\#$, $f\in [L^2(\mathbb{R}^d)]^n$ and $\epsilon>0$, we consider the elliptic problem
\begin{equation}\label{eq:ep}
  \begin{cases}
    \text{find $u_\epsilon\in [H^1(\mathbb{R}^d)]^n$ such that}
    \\ -\dive a\left(\tfrac{\cdot}{\ep}\right)\grad u_\epsilon  + s\left(\tfrac{\cdot}{\ep}\right)u_\epsilon = f.
  \end{cases}
\end{equation}
Let  $\mathcal{U}_\ep$ be the Gelfand transform, see Definition \ref{def.gt}, and $\dive_\theta$ and $\grad_\theta$ denote the divergence and gradient differential operators, respectively, on function spaces of $\theta$-quasi-periodic Sobolev functions, see Definition \ref{def.dops}. Then, the main result of the section for the class of  problems \eqref{eq:ep} is as follows.
\begin{theorem}[Fibre homogenisation theorem]\label{t:quanthom} Let $a \in \mathcal{M}_{n,d}^\#$, $s \in \mathcal{S}_{n,d}^\#$. Then, there exists  $\kappa>0$ such that for all $\epsilon>0$, the inequality
	\begin{equation*}
	\Bigl\| \big(-\dive a\left(\tfrac{\cdot}{\epsilon}\right)\grad + s\left(\tfrac{\cdot}{\epsilon}\right)\big)^{-1}  - \mathcal{U}_\epsilon^{-1}\int_{\Theta}^\oplus \big(- \epsilon^{-2}\dive_\theta a^{\textnormal{hom}}(\theta) \grad_\theta  + \m(s)\big)^{-1}d\theta\mathcal{U}_\epsilon \Bigr\| \leq \kappa \epsilon
	\end{equation*}
holds.  The constant matrix $\m(s) \in \mathcal{S}_{n,d}^\#$ and constant fourth-order tensor $a^{\textnormal{hom}}(\theta) \in \mathcal{M}_{n,d}^\#$, $\theta \in \Theta \coloneqq [-\pi,\pi)^d$, are given as follows:
\[
\m(s)_{ij}\coloneqq  \int_Y s_{ij}(y)\, {\textnormal{d}}y \qquad ({i,j\in\{1,\ldots,n\}}),
\]
and

\begin{equation}
\label{ahomtheta}
\begin{aligned}
a^{\textnormal{hom}}_{ijrs}(\theta) := \sum_{k=1}^n \sum_{l=1}^d \int_Y a_{ijkl} \big( \partial_{l} N^{(rs)}_{\theta k}(y) +e^{\i \theta \cdot y}\delta_{kr}\delta_{ls}\big)&e^{-\i \theta \cdot y}\, {\textnormal{d}}y \\ & ({i,r \in \{1,\ldots,n\}, j,s \in \{1,\ldots,d\}}),
\end{aligned}
\end{equation}
where $N^{(rs)}_\theta \in [H^1_\theta(Y) \perp e^{\i \langle\theta, \cdot \rangle_{\mathbb{C}^d}}]^n$ uniquely solves
\begin{equation}
\label{cellprob}
\langle a[\nabla N^{(rs)}_\theta + e^{\i \langle\theta, \cdot \rangle_{\mathbb{C}^d}} e_r \otimes e_s ], \nabla \phi \rangle = 0, \qquad ( \phi \in [H^1_\theta(Y) \perp e^{\i \langle\theta, \cdot \rangle_{\mathbb{C}^d}}]^n)
\end{equation}
with $e_r =(\delta_{ri})_{i\in\{1,\ldots,n\}}$ and $e_s = (\delta_{sj})_{j \in \{1,\ldots,d\}}$.
\end{theorem}
\begin{remark}
\label{r.compare} { \ } 

\begin{enumerate}[label=(\alph*)]
	\item{The well-posedness of \eqref{eq:ep} follows from noting the equivalence of this problem with a first-order formulation, see Proposition \ref{t:mtgrtheta} below, and Lemma \ref{l:invblty}.}
	\item{The well-posedness of \eqref{cellprob} is presented for the reader's convenience at the end of the section, cf. Proposition \ref{p:ahomwd}.}
	\item{It is instructive to compare the homogenisation result here to the standard result available in the literature; the standard result states that $\big(-\dive a(\cdot / \ep) \grad + s(\cdot / \ep) \big)^{-1}$ is $\ep$-close in operator-norm to $\big(-\dive a^{\textnormal{hom}} \grad+\textnormal{m}(s)\big)^{-1}$ where
		\[
		a^{\textnormal{hom}}_{ijrs} := \sum_{k=1}^n \sum_{l=1}^d\int_Y a_{ijkl} \big( \partial_{l} N^{(rs)}_{k}(y) +\delta_{kr}\delta_{ls}\big)\, {\textnormal{d}}y \qquad ({i,r \in \{1,\ldots,n\}, j,s \in \{1,\ldots,d\}}),
		\]
		for $N^{(rs)} \in [H^1_\#(Y)\perp 1]^n$ the unique solution to
		\begin{equation*}
		\langle a[\nabla N^{(rs)}_\# + e_r \otimes e_s ], \nabla \phi \rangle = 0, \qquad (\phi \in [H^1_\#(Y) \perp 1]^n).
		\end{equation*}
		A quick inspection determines the equality $a^{\textnormal{hom}} = a^{\textnormal{hom}}(\theta) \vert_{\theta = 0}$, and one can deduce that the equivalent leading-order asymptotics presented in Theorem \ref{t:quanthom} lead to the standard homogenisation result by comparing the difference $a^{\textnormal{hom}}(\theta) - a^{\textnormal{hom}}(0)$ with respect to $\theta$. This is the subject of Section \ref{sec:ahom}.
}
\end{enumerate}
\end{remark}
The remainder of this section is dedicated to the proof of Theorem \ref{t:quanthom}. The general strategy we follow is to first reformulate \eqref{eq:ep} in the framework presented in Section \ref{s:gen}; this is done in Proposition \ref{t:thetafirst}. Then we show that, in this setting, Hypothesis \ref{hyp:grtheta} \ref{en:grt1}-\ref{en:grt3} (in particular \eqref{assump1} and Hypothesis \ref{hyp:gr}) holds and, therefore, Theorem \ref{t:mtgr} follows; this is done in Propositions \ref{p:hypasse} and \ref{p:hypasse3}. Next, we show that $\widetilde{M}(\theta) = a^{\rm hom}(\theta)$ satisfies the assumptions of Theorem \ref{t:hom2}; this is identity \eqref{eq:homform}. Lastly, we aim to use Theorem \ref{t:mtgrtheta} to establish Theorem \ref{t:quanthom}. This requires proving the weak measurability assumption: Hypothesis \ref{en:grt5}; this is Theorem \ref{t:wm}. Bearing this strategy mind,  most of the work of this section will  be in  establishing Hypothesis \ref{hyp:grtheta}.

Let us begin with the reformulation of \eqref{eq:ep} via an application of the Gelfand transform:
\begin{definition}\label{def.gt}
	For $\epsilon>0$, $f\in [C_c(\mathbb{R}^d)]^n$, we define
	\[
	\mathcal{U}_\epsilon f(\theta,y)\coloneqq \left(\frac{\epsilon}{2\pi}\right)^{d/2} \sum_{k\in\mathbb{Z}^d} f\big(\epsilon(y+k)\big)e^{-i\theta\cdot k} \qquad 	(\theta\in  \Theta, y\in Y).
	\]
\end{definition}
It is well-known, see for example \cite[Section 3.2, pg. 615]{BeLiPa}, that  $\mathcal{U}_\epsilon$ extends to a unitary operator from $[L^2(\mathbb{R}^d)]^n$ into $[L^2(\Theta;L^2_\theta(\mathbb{R}^d))]^n$, where $L^2_\theta(\mathbb{R}^d)\coloneqq \{f\in L_{\textnormal{loc}}^2(\mathbb{R}^d); f(\cdot+k)=e^{\i \theta \cdot k}f(\cdot)\,(k\in\mathbb{Z}^d)\} (\cong L^2(Y))$. Henceforth, we identify $L^2(Y)$ with $L_\theta^2(\mathbb{R}^d)$.
\begin{definition}
\label{def.dops}	We define
\begin{flalign*}
\grad &\colon [H^1(\mathbb{R}^d)]^n\subseteq [L^2(\mathbb{R}^d)]^n\to [L^2(\mathbb{R}^d)]^{n\times d},(\phi_i)_{i\in\{1,\ldots,n\}}\mapsto (\partial_j\phi_i)_{i \in \{1,\ldots,n\},j\in\{1,\ldots,d\}},\\
\grad_\theta&\colon [H^1_\theta(Y)]^n\subseteq [L^2(Y)]^n\to [L^2(Y)]^{n\times d},(\phi_i)_{i\in\{1,\ldots,n\}}\mapsto (\partial_j\phi_i)_{i \in \{1,\ldots,n\},j\in\{1,\ldots,d\}},
\end{flalign*}
where $H^1_\theta(Y)$ is the Sobolev space of $\theta$-quasi-periodic functions taken to be the closure, with respect to the $H^1(Y)$ norm, of $C^\infty_\theta(Y)$: smooth functions $\phi$ that satisfy $\phi(\cdot+k)=e^{\i \theta \cdot k}\phi(\cdot)$, $k\in \mathbb{Z}^d$. We also introduce
\[
\dive \coloneqq -\grad^*,\text{ and }\dive_\theta \coloneqq -\grad^*_\theta,
\] 
as well as
\[
   \grad_\#\coloneqq \grad_0,\,\dive_\#\coloneqq \dive_0 \text{ and }H_\#^1(Y) = \dom(\grad_\#).
\]
\end{definition}
The operators just introduced are closed. Indeed, the divergence operators are skew-adjoints of the densely defined gradient operators. The operator $\grad_\theta$ is closed, since $\grad:[H^1(Y)]^n\subseteq [L^2(Y)]^n\to [L^2(Y)]^{n\times d}$ is closed, $\grad_\theta\subseteq \grad$ and $H^1_\theta(Y)\subseteq H^1(Y)$ is, by definition, closed. 

 For the convenience of the reader, we now gather some well-known properties on the interplay between the Gelfand transform and the differential operators introduced above. As is customary in PDE-theory, we employ a slight abuse of notion by not distinguishing between $\grad_\theta$ acting on $L^2(Y)$ and the corresponding gradient (acting as differentiation with respect to $y$) in $L^2\big(\Theta;L^2(Y)\big)$. 
\begin{proposition}
\label{t:guder} Let $\epsilon>0$, $a\in \mathcal{M}_{n,d}^\#$, $s\in \mathcal{S}_{n,d}^\#$. The following statements hold:
\begin{enumerate}[label=(\alph*)]
	\item\label{en:guder1} $ \mathcal{U}_\epsilon \grad = \tfrac{1}{\epsilon}\grad_\theta\mathcal{U}_\epsilon$, 
	\item\label{en:guder2} $ \mathcal{U}_\epsilon \dive = \tfrac{1}{\epsilon}\dive_\theta\mathcal{U}_\epsilon$, 
	\item\label{en:guder3} for all $(\theta,y)\in \Theta\times Y$ we have $ (\mathcal{U}_\epsilon a(\cdot/\epsilon)f)(\theta,y)=a(y)(\mathcal{U}_\epsilon f)(\theta,y)$ and \\$ (\mathcal{U}_\epsilon s(\cdot/\epsilon)f)(\theta,y)=s(y)(\mathcal{U}_\epsilon f)(\theta,y)$.
\end{enumerate}
\end{proposition}
\begin{proof}
 The proof of \ref{en:guder3} easily follows from the explicit formula for the Gelfand transformation for $f\in [C_c(\mathbb{R}^d)]^n$ and the periodicity of $a$ and $s$. The statement in \ref{en:guder2} follows from \ref{en:guder1} upon using the definition of $\dive$ and $\dive_\theta$ as, respectively, being skew-adjoints of $\grad$ and $\grad_\theta$ along with the fact $\mathcal{U}_\epsilon$ is unitary. Thus, it remains to demonstrate \ref{en:guder1}. For this, we observe that \[\mathcal{U}_\epsilon \grad \phi = \tfrac{1}{\epsilon}\grad_\theta \mathcal{U}_\epsilon\phi\] holds for $\phi\in [C_c^\infty(\mathbb{R}^d)]^n$. Therefore, we deduce $\grad\subseteq  \mathcal{U}^{-1}_\epsilon\frac{1}{\epsilon}\grad_\theta\mathcal{U}_\epsilon$ by taking into account the facts that $\grad$ and $\grad_\theta$ are closed, $\mathcal{U}_\epsilon$ is unitary, and that $[C_c^\infty(\mathbb{R}^d)]^n$ is a core for $\grad$. Similarly, as $C^\infty_{\theta}(Y)$ is a core of $\grad_\theta$ we obtain
 \[
 \mathcal{U}_\epsilon^{-1}\tfrac{1}{\epsilon}\grad_\theta \subseteq  \grad \mathcal{U}_\epsilon^{-1},
 \]
and the assertion follows.
\end{proof}
Proposition \ref{t:guder} implies that $u \in \dom\big(\dive a(\tfrac{\cdot}{\ep})\grad\big)$ solves \eqref{eq:ep} if, and only if, $\mathcal{U}_\ep u \in \dom\big(\dive_\theta  a \grad_\theta \big)$ solves
\begin{equation}
\label{probequiv}
-\tfrac{1}{\ep^2} \dive_\theta a \grad_\theta \mathcal{U}_\ep u + s \mathcal{U}_\ep u = \mathcal{U}_\ep f. 
\end{equation}
For the final step to cast the problem in the form discussed in Section \ref{s:gen}, we introduce the spaces
\begin{equation}\label{Pspace}
P(\theta)\coloneqq e^{\i \langle\theta, \cdot \rangle_{\mathbb{C}^d}}\mathbb{C}^{n \times d} \oplus \big\{ \grad_\theta u \, ; \, u \in [H^1_\theta(Y)\perp e^{\i \langle\theta, \cdot \rangle_{\mathbb{C}^d}}]^n \big\} \quad (\theta \in \Theta),
\end{equation}
where here, and throughout, $e^{i\langle\theta, \cdot \rangle_{\mathbb{C}^d}}\mathbb{C}^{n\times d}$ is the space obtained by multiplying each component of vectors in $\mathbb{C}^{n \times d}$ by $Y\ni y\mapsto e^{i \langle\theta, y \rangle_{\mathbb{C}^d}}$. 

In order to properly establish and formulate the first order perspective we have in mind we first demonstrate that  $\rge(\grad_\theta)$ and $P(\theta)$ are closed. Both results are a consequence of the following standard argument. 

\begin{lemma}\label{t:closedrange} Let $H_0,H_1$ be Hilbert spaces and $B\colon \dom(B)\subseteq H_0\to H_1$ closed. Assume that $B$ is one-to-one and $\dom(B)\hookrightarrow H_0$ is compact. Then, there exists $c>0$ such that
\[
   \|\phi\|_{H_0}\leq c\|B\phi\|_{H_1}.
\]
In particular, $\rge(B)\subseteq H_1$ is closed.
\end{lemma}
\begin{proof}
 Assume that the inequality does not hold for any positive constant.  Then, there exists a sequence 
 $(\phi_k)_{k\in\mathbb{N}}$ such that $\|\phi_k\|_{H_0}=1$ and 
 \[
  \|B\phi_k\|_{H_1} < \tfrac{1}{k}\quad (k\in\mathbb{N}).
 \]
 As $(\phi_k)_{k\in\mathbb{N}}$ is bounded in $\dom(B)$, and $\dom(B)\hookrightarrow H_0$ is compact, we deduce that there exists a $H_0$-convergent subsequence of $(\phi_k)_k$ with $(B\phi_k)_k$ weakly converging, which we do not relabel. Let $\phi\coloneqq \lim_{k\to\infty} \phi_k\in H_0$. By passing to the limit  $k\to\infty$, in the inequality $\|B\phi_k\|_{H_1}< 1/k$, we deduce that 
 \[
 \Vert \text{ (weak)-}\lim_{k\to \infty }  B \phi_k \Vert_{H_1} \le \liminf_{k \rightarrow \infty} \Vert B \phi_k \Vert_{H_1} =0,
 \]
and therefore  $\phi\in \dom(B)$ with $B\phi=0$. As $B$ is one-to-one, it follows that  $\phi=0$ which contradicts $\|\phi\|_{H_0}=\lim_{k\to\infty}\|\phi_k\|_{H_0}=1$. Hence, the desired inequality holds.
 
 The fact that the range of $B$ is closed is a straightforward consequence of the now established inequality and the fact that $B$ is closed.
\end{proof}

\begin{proposition}\label{p:clo} Let $\theta\in \Theta = [-\pi,\pi)^d$. Then, the following assertions hold:
\begin{enumerate}[label=(\alph*)]
 \item\label{clo1} $\rge(\grad_\theta)\subseteq [L^2(Y)]^{n\times d}$ is closed,
 \item\label{clo2} $P(\theta)\subseteq [L^2(Y)]^{n\times d}$, introduced in \eqref{Pspace}, is closed.
\end{enumerate}
\end{proposition}
\begin{proof}
Note that $\rge(\grad_\theta)=\rge(\grad_\theta|_{\kar(\grad_\theta)^{\bot}})$. To establish \ref{clo1} we aim to apply Lemma \ref{t:closedrange} for $B= \grad_\theta|_{\kar(\grad_\theta)^\bot}$, $H_0=[L^2(Y)]^n$ and $H_1=[L^2(Y)]^{n\times d}$.  $B$ is easily shown to be one-to-one and closed. By Rellich's selection theorem $H^1(Y)\hookrightarrow L^2(Y)$ is compact. Hence, since $\dom(B)=H^1_\theta(Y)\cap \kar(\grad_\theta)^\perp\subseteq H^1(Y)$ is closed, we deduce that $\dom(B)\hookrightarrow L^2(Y)$ is compact. Thus, \ref{clo1} follows from Lemma \ref{t:closedrange}. 
 
 In order to prove \ref{clo2}, we observe that $e^{\i \langle\theta, \cdot \rangle_{\mathbb{C}^d}}\mathbb{C}^{n \times d}$ is finite-dimensional. Thus, we are left with proving that 
\[
 \big\{ \grad_\theta u \, ; \, u \in [H^1_\theta(Y)\perp e^{\i \langle\theta, \cdot \rangle_{\mathbb{C}^d}}]^n \big\}
\]
is closed. We demonstrated above that Lemma \ref{t:closedrange} holds for $B= \grad_\theta|_{\kar(\grad_\theta)^\bot}$, $H_0=[L^2(Y)]^n$ and $H_1=[L^2(Y)]^{n\times d}$. Consequently, the inequality in Lemma \ref{t:closedrange} holds and,  to prove the above space is closed, we only need to establish that if $(\phi_k)_{k\in\mathbb{N}}$ is a convergent sequence in $H^1_\theta(Y)$ with limit $\phi\in H^1_\theta(Y)$ satisfying
\[
   \langle \phi_k,e^{\i \langle \theta,\cdot\rangle_{\mathbb{C}^d}}\rangle_{L^2(Y)}=0 \quad (k \in \N),
\]
then  $\phi\perp e^{\i \langle \theta,\cdot\rangle_{\mathbb{C}^d}}$. This is an easy consequence of the fact that $(\phi_k)_{k\in\mathbb{N}}$ strongly converges in $L^2(Y)$. 
\end{proof}

We introduce 
\[ \iota_\theta \colon P(\theta) \hookrightarrow [L^2(Y)]^{n\times d}  \quad(\theta\in \Theta).
\]
By Proposition \ref{p:clo} \ref{clo2} we have that $P(\theta) \subseteq [L^2(Y)]^{n \times d}$ introduced in \eqref{Pspace} is closed. Thus,  $\iota_\theta^*: [L^2(Y)]^{n\times d}  \rightarrow P(\theta)$ is the well-defined adjoint operator and $\pi_{P(\theta)} \coloneqq \iota_\theta \iota_\theta^*$ is the orthogonal projection onto $P(\theta)$. The following result holds.
\begin{proposition}\label{t:thetafirst} Let $\epsilon>0$,  $\theta\in \Theta$, $a \in \mathcal{M}_{n,d}^\#$, $s \in \mathcal{S}_{n,d}^\#$ and $g\in [L^2(Y)]^n$. Then, the following conditions are equivalent:
\begin{enumerate}[label=(\roman*)]
 \item\label{en:fo1}$u\in \dom\big(\dive_\theta a \grad_\theta \big)$ satisfies
 \[
    -\tfrac{1}{\epsilon^2} \dive_\theta  a \grad_\theta u + su=g.
 \] 
 \item\label{en:fo2} $u\in \dom(\grad_\theta)$ and $q\in \dom(\dive_\theta \iota_{\theta})$ satisfy
 \[
    \left[\begin{pmatrix}
      s & 0 \\ 0& (\iota_{\theta}^*a\iota_\theta)^{-1}
    \end{pmatrix}+ \frac{1}{\epsilon}\begin{pmatrix} 0 & - \dive_\theta\iota_\theta \\ 
                                         -\iota_\theta^* \grad_\theta &0 \end{pmatrix}	\right]\begin{pmatrix} u\\ q
                                         \end{pmatrix}=\begin{pmatrix} g\\ 0
                                         \end{pmatrix}
 \]
\end{enumerate} 
\end{proposition}
\begin{proof}
 Before we prove the equivalence, we note that
 \begin{equation}\label{eq:fo}
    \dive_\theta  a \grad_\theta=\dive_\theta\iota_\theta\iota_\theta^* a\iota_\theta\iota_\theta^* \grad_\theta, \quad (\theta \in \Theta).
 \end{equation}
Indeed, note that $\rge(\grad_\theta)\subseteq P(\theta)$: This is obvious for $\theta = 0$, so let us consider $\theta \neq 0$. Since $\{ e^{\i \langle( \theta + 2\pi z), \cdot \rangle_{\mathbb{C}^d}}\}_{z \in \mathbb{Z}^d}$ forms a complete orthonormal system for $L^2(Y)$, then 
 \[
 u(y) = \sum_{z \in \mathbb{Z}^d} c^{(z)} e^{\i \langle(\theta +2\pi z),y\rangle} \quad (y \in Y, c^{(z)} \in \mathbb{C}^n),
 \]
 and 
 \[
 \grad_\theta u(y)= \sum_{z \in \mathbb{Z}^d} c^{(z)} \otimes ( \i \theta + 2\pi z) e^{\i \langle(\theta +2\pi z),y\rangle} = c^{(0)} \otimes \i \theta e^{\i \langle\theta, y\rangle} + \grad_\theta v(y) , \quad (y \in Y, c^{(0)} \in \mathbb{C}^n),
 \]
for some $v \in [H^1_\theta(Y)\perp e^{\i \langle\theta, \cdot \rangle_{\mathbb{C}^d}}]^n$.

\noindent  Next, as $\kar(\dive_\theta)^\bot=\rge(\grad_\theta)\subseteq P(\theta)$, we obtain $P(\theta)^\bot\subseteq \kar(\dive_\theta)$. In particular, we infer
 \[
    \iota_\theta\iota_\theta^* \grad_\theta=\grad_\theta\text{ and }\dive_\theta(1-\iota_\theta\iota_\theta^*)=0.
 \]
 Hence, \eqref{eq:fo} follows.
 
 For \ref{en:fo1}$\Rightarrow$\ref{en:fo2}, we set $q\coloneqq \tfrac{1}{\epsilon}\iota_\theta^* a\iota_\theta\iota_\theta^* \grad_\theta u$. Then \ref{en:fo2} follows from \eqref{eq:fo}. Note that, for $a \in \mathcal{M}_{n,d}^{\#}$, $\iota_\theta^* a\iota_\theta$ is continuously invertible. Indeed, multiplication with $a$ can be identified as an operator in $L(L^2(Y)^{n\times d})$. Moreover, as  $a \in \mathcal{M}_{n,d}^{\#}$ then   $\Re a\geq \nu 1_{[L^2(Y)]^{n\times d}}$ and consequently  $\Re\iota_\theta^*a\iota_\theta\geq \nu 1_{P(\theta)}$ for some $\nu>0$. This yields the continuous invertibility of $\iota_\theta^*a\iota_\theta$. 
 
 The implication \ref{en:fo2}$\Rightarrow$\ref{en:fo1} also follows from \eqref{eq:fo}. Note that $u\in \dom(\dive_\theta a \grad_\theta)$ follows from the fact that $q\in \dom(\dive_\theta \iota_\theta)$, $u\in \dom(\iota_\theta^* \grad_\theta)$ and the second row of the system \ref{en:fo2}.
\end{proof}

Now, we aim to apply Theorem \ref{t:mtgrtheta} to the system \ref{en:fo2} in Proposition \ref{t:thetafirst}. For this, we use the following setting:
\begin{alignat}{1}\label{block:ell}
\begin{aligned}
   H_0 &= L^2(\Theta\times Y)^{2n + d}, &   \Theta &= [-\pi,\pi)^d,
   \\ H_\theta & = [L^2(Y)]^n\oplus P(\theta), &  H & = [L^2(Y)]^n\oplus [L^2(Y)]^{n\times d},
   \\ M(\theta) & = \begin{pmatrix}
      s & 0 \\ 0& (\iota_{\theta}^*a\iota_\theta)^{-1}
    \end{pmatrix}, &  A(\theta) & = \begin{pmatrix} 0 & -\dive_\theta \iota_\theta \\ 
                                         -\iota_\theta^* \grad_\theta &0 \end{pmatrix},
    \\ N(\theta) & = e^{\i \langle\theta, \cdot \rangle_{\mathbb{C}^d}}\mathbb{C}^n \oplus e^{\i \langle\theta, \cdot \rangle_{\mathbb{C}^d}}\mathbb{C}^{n \times d}.
\end{aligned}
\end{alignat}
We also set
\begin{equation}\label{eq:n12}
  N_1(\theta)\coloneqq e^{\i \langle\theta, \cdot \rangle_{\mathbb{C}^d}} \mathbb{C}^n,\text{ and }N_2(\theta)\coloneqq e^{\i \langle\theta, \cdot \rangle_{\mathbb{C}^d}}\mathbb{C}^{n \times d} \end{equation}

The following result holds.

\begin{theorem}\label{t:hypasse} With the setting \eqref{block:ell}, Hypothesis \ref{hyp:grtheta} holds.
 \end{theorem}
 We begin with verifying the conditions \ref{en:grt1} and \ref{en:grt2} of Hypothesis \ref{hyp:grtheta} as well as \ref{gr1} and \ref{gr2} of Hypothesis \ref{hyp:gr}.
 \begin{proposition}\label{p:hypasse} Assume the setting \eqref{block:ell}. For each $\theta\in \Theta$, the following statements hold:
 	\begin{enumerate}
 		\item $A(\theta)$ is skew-selfadjoint;
 		\item $\Re M(\theta)\geq \nu/(\|a\|^2+1)$, where $\nu>0$ is such that $\Re a\geq \nu 1_{\mathbb{C}^{n \times d}}$ and $\Re s\geq \nu 1_{\mathbb{C}^{n}}$;
 		\item $\pi_{N(\theta)}A(\theta)\subseteq A(\theta)\pi_{N(\theta)}$;
 		\item $A(\theta)\pi_{N(\theta)}$ is bounded.
 	\end{enumerate}  
 \end{proposition}
 \begin{proof}
 	The first assertion follows from the fact that $\dive_\theta \iota_\theta=- \big(\iota_\theta^* \grad_\theta\big)^*$. For the second statement, we observe that $\Re s\geq \nu1_{\mathbb{C}^{n}}\geq \big(\nu/(\|a\|^2+1)\big)1_{\mathbb{C}^{n}}$. Moreover, note that $\Re a\geq \nu 1_{\mathbb{C}^{n\times d}}$ implies $\iota_\theta^*a\iota_\theta\geq \nu 1_{P(\theta)}$ and, thus,  
 	$$
 	\Re (\iota_\theta^*a\iota_\theta)^{-1} \geq \big(\nu/(\|\iota_\theta^* a\iota_\theta\|^2)\big)1_{P(\theta)}\geq \big(\nu/(\|a\|^2+1)\big)1_{P(\theta)}.
 	$$ The third assertion is easy to see upon the decomposition $[L^2(Y)]^n = e^{\i \langle\theta, \cdot\rangle_{\mathbb{C}^d}}\mathbb{C}^n\oplus [L^2(Y)\perp e^{\i \langle\theta, \cdot\rangle_{\mathbb{C}^d}}]^n $. The fourth assertion is a consequence of the above decomposition of $[L^2(Y)]^n$ and the finite dimensionality of $N(\theta)$.
 \end{proof}
\begin{proof}[Proof of Theorem \ref{t:hypasse} -- Part 1]
  The assertions \ref{en:grt1} and \ref{en:grt2} of Hypothesis \ref{hyp:grtheta} and \ref{gr1} of Hypothesis \ref{hyp:gr}  clearly follow from Proposition \ref{p:hypasse}. Assertion Hypothesis \ref{hyp:gr} \ref{gr2} follows from Proposition \ref{p:hypasse} (c) and Lemma \ref{l:Athetainv} upon setting $H=H_\theta$, $U=N(\theta)$ and $A=A(\theta)$.
\end{proof}

We now turn to complete the proof of \ref{en:grt3} of Hypothesis \ref{hyp:grtheta}, which results from a quantified version of Proposition \ref{p:clo} (see also Lemma \ref{t:closedrange}).

\begin{proposition}\label{p:hypasse3}Assume the setting \eqref{block:ell}. Then, the following assertions hold.
	\begin{enumerate}[label=(\alph*)]
		\item\label{en:h31} For all $\theta\in \Theta$ and $u\in [H^1_\theta(Y)\perp e^{\i\langle\theta, \cdot\rangle_{\mathbb{C}^d}}]^n$ we have
		\[
		\|u\|_{[L^2(Y)]^n}\leq \pi^{-1} \|\grad_\theta u\|_{[L^2(Y)]^{n \times d}}.
		\]
		\item\label{en:h31.5} For all $\theta\in \Theta$, we have \[R(\theta)=N(\theta)^{\bot}= (e^{\i \langle\theta, \cdot\rangle_{\mathbb{C}^d}}\mathbb{C}^{n})^\bot\oplus \{\grad_\theta u; u \in [H^1_\theta(Y) \perp e^{\i \langle\theta, \cdot\rangle_{\mathbb{C}^d}}]^n \}.\]
		\item\label{en:h32} Let $\iota_{R(\theta)}: R(\theta) \hookrightarrow H_\theta$ be the canonical embedding. For all $\theta\in \Theta$, the operator $\iota_{R(\theta)}^*A(\theta)\iota_{R(\theta)}$ is continuously invertible and
		\[
		\sup_{\theta\in \Theta} \left\| (\iota_{R(\theta)}^*A(\theta)\iota_{R(\theta)})^{-1} \right\|\leq \pi^{-1}.
		\]
	\end{enumerate}
\end{proposition}
\begin{proof} 
To prove \ref{en:h31}, we argue, as in the proof of Proposition \ref{t:thetafirst}, that $\{  e^{\i \langle\theta+2\pi z, \cdot\rangle_{\mathbb{C}^d}} \}_{z \in \mathbb{Z}^d}$ is an orthonormal basis for $L^2(Y)$ and utilising the fact that $u_i \perp e^{\i \langle\theta, \cdot\rangle_{\mathbb{C}^d}}$, $i\in\{1,\ldots,n\}$, one has
\[
\begin{aligned}
u = \sum_{\substack{z \in \mathbb{Z}^d \\ z \neq 0}} c^{(z)} e^{\i \langle\theta+2\pi z, \cdot\rangle_{\mathbb{C}^d}}, & \quad & \grad_\theta u = \sum_{\substack{z \in \mathbb{Z}^d \\ z \neq 0}}  e^{\i \langle\theta+2\pi z, \cdot\rangle_{\mathbb{C}^d}} c^{(z)}\otimes \i(\theta + 2\pi z ), \quad c^{(z)} \in \mathbb{C}^n.
\end{aligned}
\]
Then 
\[
\Vert \grad_\theta u \Vert_{[L^2(Y)]^{n\times d}}^2 = \sum_{\substack{z \in \mathbb{Z}^d \\ z \neq 0}}	\Vert c^{(z)}\otimes \i(\theta + 2\pi z ) \Vert_{\mathbb{C}^{n \times d}}^2 \ge \pi^2 \sum_{\substack{z \in \mathbb{Z}^d \\ z \neq 0}}	\Vert c^{(z)} \Vert_{\mathbb{C}^{n}}^2
 =\pi^2\Vert u \Vert_{[L^2(Y)]^n}^2.
\]
 The statement in \ref{en:h31.5} immediately follows from the definition of $P(\theta)$, see \eqref{Pspace}.
	For the proof of statement \ref{en:h32}, we set
\[
\begin{aligned}
R_1(\theta)\coloneqq (e^{\i \langle\theta, \cdot\rangle_{\mathbb{C}^d}}\mathbb{C}^n)^\perp, &  \quad \text{ and } & R_2(\theta)\coloneqq \{\grad_\theta u;u\in [H_\theta^1(Y)\bot e^{\i \langle\theta, \cdot\rangle_{\mathbb{C}^d}}]^n\}.
\end{aligned}
\]
Thus, for the canonical embeddings $\iota_{R_1(\theta)} : R_1(\theta) \hookrightarrow [L^2(Y)]^n$, $\iota_{R_2(\theta)}: R_2(\theta) \hookrightarrow P(\theta)$, we obtain
	\[
	\iota_{R(\theta)} = \begin{pmatrix} \iota_{R_1(\theta)} & 0 \\ 0& \iota_{R_2(\theta)} \end{pmatrix},
	\] 
and
	\begin{align*}
	\iota_{R(\theta)}^* A(\theta) \iota_{R(\theta)} & 
	 =\begin{pmatrix} 0 & -\iota_{R_1(\theta)}^*\dive_\theta\iota_\theta\iota_{R_2(\theta)} \\ 
	-\iota_{R_2(\theta)}^*\iota_\theta^*\grad_\theta\iota_{R_1(\theta)} &0 \end{pmatrix}.
	\end{align*}
	Next, we observe that $\iota_{R_2(\theta)}^*\iota_\theta^*$ projects onto $R_2(\theta)$. By \ref{en:h31} it follows that  $\grad_\theta\iota_{R_1(\theta)}$ is one-to-one, and therefore we obtain that $\iota_{R_2(\theta)}^*\iota_\theta^*\grad_\theta\iota_{R_1(\theta)}$ is a bijection. Therefore, it follows that  \[\left(\iota_{R_2(\theta)}^*\iota_\theta^*\grad_\theta\iota_{R_1(\theta)}\right)^*=-\iota_{R_1(\theta)}^*\dive_\theta\iota_\theta\iota_{R_2(\theta)}\] 
is a bijection.	In particular, by \ref{en:h31}, we calculate
	\[
	\big\Vert \big(\iota_{R_1(\theta)}^*\dive_\theta\iota_\theta\iota_{R_2(\theta)}\big)^{-1}\big\Vert=\big\|\big(\iota_{R_2(\theta)}^*\iota_\theta^*\grad_\theta\iota_{R_1(\theta)}\big)^{-1}\big\|\leq \pi^{-1}.
	\]
	Hence, 
	\[
	\big\| \big(\iota_{R(\theta)}^* A(\theta) \iota_{R(\theta)}\big)^{-1}\big\|\leq \pi^{-1},
	\]and we conclude the proof of assertion \ref{en:h32}.
\end{proof}
\begin{proof}[Proof of Theorem \ref{t:hypasse} -- Part 2] 
	The assertion \ref{en:grt3} of Hypothesis \ref{hyp:grtheta} follows from Proposition \ref{p:hypasse3}\ref{en:h32}. 
\end{proof}
To complete the proof of Theorem \ref{t:hypasse}, it remains to prove Hypothesis \ref{en:grt5}. For this, we make some preliminary observations.
 The proof of the next result is demonstrated by direct calculation and is therefore omitted.
\begin{proposition}
\label{p:proj}
Let $\phi \in [L^2(Y)]^{n \times d}$, $\theta \in \Theta$. Then,
\[
\pi_{P(\theta)} \phi = \langle \phi, e^{\i \langle \theta, \cdot \rangle_{\C^d}} \rangle  e^{\i \langle \theta, \cdot \rangle_{\C^d}} - \sum_{z \in \Z^d \backslash \{ 0 \}} \frac{(\theta + 2\pi z)(\theta + 2\pi z)^T}{\vert \theta + 2\pi z \vert^2} c^{(z)}_\phi  e^{\i \langle \theta+ 2\pi z, \cdot \rangle_{\C^d}},
\]
where
\[
c^{(z)}_\phi \coloneqq \langle \phi,  e^{\i \langle \theta+2\pi z, \cdot \rangle_{\C^d}} \rangle .
\]
\end{proposition}

\begin{proposition}
 \label{p:athetacont} Let $(\theta_k)_{k\in \mathbb{N}}$ be a convergent sequence in $\Theta$, $\theta\coloneqq \lim_{k\to\infty}\theta_k$. Let $(u_k)_{k\in\mathbb{N}}$ in $[L^2(Y)]^n$, and $(q_k)_{k\in\mathbb{N}}$ in $[L^2(Y)]^{n\times d}$ weakly convergent sequences with limits $u$ and $q$. 
 Then, the following assertions hold.
 \begin{enumerate}[label=(\alph*)]
  \item\label{ath1} $\pi_{P(\theta_k)}q_k\rightharpoonup \pi_{P(\theta)}q$.
  \item\label{ath2} Assume, in addition, that $q_k\in P(\theta_k)$, $k\in\mathbb{N}$, as well as $(\grad_{\theta_k} u_k)_{k\in\mathbb{N}}$ and $(\dive_{\theta_k} q_k)_{k\in\mathbb{N}}$ are bounded. Then $u\in\dom(\grad_\theta)$, $q\in\dom(\dive_\theta)$ and 
  \[ 
      \grad_{\theta_k}u_k\rightharpoonup \grad_\theta u,\text{ and }\dive_{\theta_k}q_k\rightharpoonup \dive_\theta q.
  \]
 \end{enumerate}
\end{proposition}
\begin{proof}
 For the proof of \ref{ath1}, we use Proposition \ref{p:proj}. Indeed, we obtain for all $k\in\mathbb{N}$ with $c^{(z)}_{q_k}= \langle q_k,  e^{\i \langle \theta_k+2\pi z, \cdot \rangle_{\C^d}} \rangle$ and that
 \begin{align*}
      \pi_{P(\theta_k)}q_k
      & = \langle q_k, e^{\i \langle \theta_k, \cdot \rangle_{\C^d}} \rangle  e^{\i \langle \theta_k, \cdot \rangle_{\C^d}} - \sum_{z \in \Z^d \backslash \{ 0 \}} \frac{(\theta_k + 2\pi z)(\theta_k + 2\pi z)^T}{\vert \theta_k + 2\pi z \vert^2} c^{(z)}_{q_k}  e^{\i \langle \theta_k+ 2\pi z, \cdot \rangle_{\C^d}}.
 \end{align*}
 As $(q_k)_k$ converges weakly to $q$, we obtain that 
 \[
    \frac{(\theta_k + 2\pi z)(\theta_k + 2\pi z)^T}{\vert \theta_k + 2\pi z \vert^2} c^{(z)}_{q_k}  e^{\i \langle \theta_k+ 2\pi z, \cdot \rangle_{\C^d}}\to
    \frac{(\theta + 2\pi z)(\theta + 2\pi z)^T}{\vert \theta + 2\pi z \vert^2} c^{(z)}_{q}  e^{\i \langle \theta+ 2\pi z, \cdot \rangle_{\C^d}}
 \]
 as $k\to\infty$. Thus, by the dominated convergence theorem, we infer
 \[
    \pi_{P(\theta_k)}q_k\rightharpoonup \langle q, e^{\i \langle \theta, \cdot \rangle_{\C^d}} \rangle  e^{\i \langle \theta_, \cdot \rangle_{\C^d}} - \sum_{z \in \Z^d \backslash \{ 0 \}} \frac{(\theta + 2\pi z)(\theta + 2\pi z)^T}{\vert \theta + 2\pi z \vert^2} c^{(z)}_{q}  e^{\i \langle \theta+ 2\pi z, \cdot \rangle_{\C^d}}=\pi_{P(\theta)}q.
 \]
 Hence, \ref{ath1} follows.
 
 The second statement is proved in a similar manner and so we will just sketch the argument. Upon decomposing $u_k$ with respect to the basis $\{ e^{\i \langle \theta + 2\pi z, \cdot\rangle_{\C^d}} \}_{z \in \Z^d}$, decomposing $q_k$ as above,  one computes 
 \[
 \grad_{\theta_k} u_k = \sum_{z \in \Z^d} \i (\theta_k +2\pi z)\otimes c^{(z)}_{u_k} e^{\i \langle \theta_k + 2\pi z, \cdot \rangle_{\C^d}},
 \]
 \[
 \dive_{\theta_k} q_k = \i \theta_k^T \langle q_k, e^{\i \langle \theta_k, \cdot \rangle_{\C^d}} \rangle  e^{\i \langle \theta_k, \cdot \rangle_{\C^d}} - \sum_{z \in \Z^d \backslash \{ 0 \}} (\theta_k + 2\pi z)^T \frac{(\theta_k + 2\pi z)(\theta_k + 2\pi z)^T}{\vert \theta_k + 2\pi z \vert^2} c^{(z)}_{q_k}  e^{\i \langle \theta_k+ 2\pi z, \cdot \rangle_{\C^d}}.
 \]
Then,  utilising the assumption that both the sequences $(\grad_{\theta_k} u_k )_{k \in \N}$ and $( \dive_{\theta_k} q_k )_{k \in \N}$ are bounded, we can pass to the limit in the above equations and characterise them as $\grad_\theta u$ and $\dive_\theta q$ respectively.
\end{proof}

\begin{theorem}
\label{t:wm} 
Let $s \in \mathcal{S}^{\#}_{n,d}$, $a \in \mathcal{M}^{\#}_{n,d}$, and $\ep \in (0,\infty)$. Assume setting \eqref{block:ell}. Consider $T : \Theta \rightarrow L(H)$ be given by 
\[
\theta \mapsto \tilde{\iota}_\theta \left( \begin{pmatrix} s & 0 \\ 0& (\iota_{\theta}^*a\iota_\theta)^{-1}\end{pmatrix} + \frac{1}{\epsilon}\begin{pmatrix} 0 & -\dive_\theta\iota_\theta \\ -\iota_\theta^* \grad_\theta & 0\end{pmatrix}\right)^{-1}\tilde{\iota_\theta}^*.
\]
Then, $T$ is weakly continuous.
\end{theorem}
\begin{proof}
 Before we prove the statement, we observe that there exists $c>0$ such that for all $\theta\in\Theta$ one has
 \[
    \Re s\geq c 1_{\C^n},\text{ and }\Re {\iota_\theta^* a\iota_\theta} \geq c 1_{P(\theta)}.
 \]
 Hence, by Lemma \ref{l:invblty}, we deduce that
 \begin{equation}\label{eq:wm1}
    \left\|\left( \begin{pmatrix} s & 0 \\ 0& (\iota_{\theta}^*a\iota_\theta)^{-1}\end{pmatrix} + \frac{1}{\epsilon}\begin{pmatrix} 0 & -\dive_\theta\iota_\theta \\ -\iota_\theta^* \grad_\theta & 0\end{pmatrix}\right)^{-1}\tilde{\iota_\theta}^*\right\|\leq \tfrac{1}{c}.
 \end{equation}
 Moreover, it is clear that
  \begin{equation}\label{eq:wm2}
 \begin{aligned}
   \sup_{\theta\in\Theta}\left\|\begin{pmatrix} 0 & -\dive_\theta\iota_\theta \\ -\iota_\theta^* \grad_\theta & 0\end{pmatrix}\left( \begin{pmatrix} s & 0 \\ 0& (\iota_{\theta}^*a\iota_\theta)^{-1}\end{pmatrix} + \frac{1}{\epsilon}\begin{pmatrix} 0 & -\dive_\theta\iota_\theta \\ -\iota_\theta^* \grad_\theta & 0\end{pmatrix}\right)^{-1}\tilde{\iota_\theta}^*\right\| \\ < \infty. 
 \end{aligned}
 \end{equation}
 For the proof of the statement, we let $(\theta_k)_{k\in\mathbb{N}}$ be a convergent sequence in $\Theta$; denote by $\theta$ its limit. Let $f\in [L^2(Y)]^n$, $g\in [L^2(Y)]^{n\times d}$ and define $(u_k,q_k)\coloneqq T(\theta_k)(f,g)$. Then, by \eqref{eq:wm1} and \eqref{eq:wm2}, we obtain that
 $(u_k)_k$, $(q_k)_k$, $(\dive_\theta\iota_\theta q_k)_k$, and $(\grad_\theta u_k)_k$ are bounded. Without loss of generality, we may assume that $(u_k)_k$ and $(q_k)_k$ converge weakly to some $u$ and $q$ respectively. Thus, by the definition of $u_k$ and $q_k$, we obtain for all $k\in\mathbb{N}$ that
 \begin{align*}
   f & = s u_k -\dive_{\theta_k}\iota_{\theta_k} q_k,
   \\ \iota_{\theta_k}^*a\iota_{\theta_k}\pi_{P(\theta_k)} g & = q_k -  \iota_{\theta_k}^*a\iota_{\theta_k}\iota_{\theta_k}^*\grad_{\theta_k}u_k.
 \end{align*}
 By Proposition \ref{p:athetacont}, as $k\to\infty$, we obtain that the weak limits of the above equations are
 \begin{align*}
   f & = s u -\dive_\theta\iota_\theta q,
   \\ \iota_{\theta}^*a\iota_{\theta}\pi_{P(\theta)} g & = q -  \iota_{\theta}^*a\iota_{\theta}\iota_{\theta}^*\grad_{\theta}u.
 \end{align*}
 These in turn imply that $(u,q)=T(\theta)(f,g)$ which identifies the limit and the assertion follows.
\end{proof}

\begin{remark} With a rationale similar to the one used in \cite{tEGoWa} and utilising that the embedding $H_\theta^1(Y)\hookrightarrow L^2(Y)$ is compact, it can be shown that the mapping in Theorem \ref{t:wm} is even continuous in operator-norm.
\end{remark}

\begin{proof}[Proof of Theorem \ref{t:hypasse} -- Part 3] 
 It remains to prove assertion \ref{en:grt5} of Hypothesis \ref{hyp:grtheta}. This is true as $\theta \mapsto \iota_\theta \big( M(\theta) + \tfrac{1}{\ep} A(\theta) \big)^{-1} \iota_\theta^*$ is weakly continuous, see Theorem \ref{t:wm}, and, therefore, weakly measurable.
\end{proof}

We are now in the position to provide a proof of the main result of this section.

\begin{proof}[Proof of Theorem \ref{t:quanthom}]
Let $\tilde{\iota}_\theta:[L^2(Y)]^n\oplus P(\theta) \hookrightarrow [L^2(Y)]^n\oplus[L^2(Y)]^{n \times d}$. Theorem \ref{t:hypasse} implies that the assumptions of Theorem \ref{t:mtgrtheta} hold for the setting \eqref{block:ell}. Therefore, we deduce that there exists a $\kappa>0$ such that for all $\epsilon>0$, we obtain
		\begin{equation}\label{ffres}
		\begin{aligned}
	&	\Bigl\| \int_{\Theta}^\oplus \tilde{\iota}_\theta \left( \begin{pmatrix} s & 0 \\ 0& (\iota_{\theta}^*a\iota_\theta)^{-1}\end{pmatrix} + \frac{1}{\epsilon}\begin{pmatrix} 0 & -\dive_\theta\iota_\theta \\ -\iota_\theta^* \grad_\theta & 0\end{pmatrix}\right)^{-1}\tilde{\iota_\theta}^*d\theta \\& - \int_{\Theta}^\oplus \tilde{\iota}_\theta \left( \pi_{N(\theta)}\begin{pmatrix} s & 0 \\ 0& (\iota_{\theta}^*a\iota_\theta)^{-1}\end{pmatrix}\pi_{N(\theta)} + \frac{1}{\epsilon}\begin{pmatrix} 0 & -\dive_\theta\iota_\theta \\ -\iota_\theta^* \grad_\theta & 0\end{pmatrix}\right)^{-1}\tilde{\iota_\theta}^*d\theta  \Bigr\|\\
	& \hspace{.8\textwidth} \leq \kappa \epsilon.
		\end{aligned}
		\end{equation}
	We shall prove below the homogenisation formulae
	\begin{equation}\label{eq:homform}
	\iota^*_{N(\theta)}\begin{pmatrix} s & 0 \\ 0& (\iota_{\theta}^*a\iota_\theta)^{-1}\end{pmatrix}\iota_{N(\theta)} = \begin{pmatrix} \m(s) & 0 \\ 0& a^\textnormal{hom}(\theta)^{-1} \end{pmatrix}.
	\end{equation}
Now, clearly the right-hand side of \eqref{eq:homform} satisfies the assumptions of Theorem \ref{t:hom2} and we deduce that
		\begin{multline}\label{ffresss}
		\Bigl\| \int_{\Theta}^\oplus \tilde{\iota}_\theta \left( \begin{pmatrix} \m(s) & 0 \\ 0& a^\textnormal{hom}(\theta)^{-1} \end{pmatrix} + \frac{1}{\epsilon}\begin{pmatrix} 0 & -\dive_\theta\iota_\theta \\ -\iota_\theta^* \grad_\theta & 0\end{pmatrix}\right)^{-1}\tilde{\iota}_\theta^*d\theta \\ - \int_{\Theta}^\oplus \tilde{\iota}_\theta \left( \pi_{N(\theta)}\begin{pmatrix} s & 0 \\ 0& (\iota_{\theta}^*a\iota_\theta)^{-1}\end{pmatrix}\pi_{N(\theta)} + \frac{1}{\epsilon}\begin{pmatrix} 0 & -\dive_\theta\iota_\theta \\ -\iota_\theta^* \grad_\theta & 0\end{pmatrix}\right)^{-1}\tilde{\iota}_\theta^*d\theta  \Bigr\|\\ \leq \kappa \epsilon.
		\end{multline}
	The above assertions prove the desired result. Indeed, after having applied the unitary Gelfand transformation, Proposition \ref{t:guder} implies the equivalence of problems \eqref{eq:ep} and \eqref{probequiv}. Then, Proposition \ref{t:thetafirst} establishes the equivalence between the first and second-order formulations, and finally \eqref{ffres}, \eqref{ffresss} imply the required asymptotics for the first-order problem.

It remains to prove \eqref{eq:homform}. We use $N(\theta)=N_1(\theta)\oplus N_2(\theta)$, see \eqref{eq:n12}. First, we establish that
\begin{equation}
\label{p:homs}		\pi_{N_1(\theta)}s\pi_{N_1(\theta)} = \m(s)=\int_Y s(y)dy \quad (\theta \in \Theta).
\end{equation}		
This is a simple calculation:
\[
\langle s e^{\i \langle\theta,\cdot\rangle_{\mathbb{C}^d}} \alpha, e^{\i \langle\theta,\cdot\rangle_{\mathbb{C}^d}} \beta \rangle =\sum_{i,j\in\{1,\ldots,n\}} \left( \int_{Y}s_{ij}\right)\alpha_i \beta_j, \quad ( \alpha, \beta \in \mathbb{C}^n).
\]
Let us now prove that
\begin{equation}
\label{homforma}
\iota_{N_2(\theta)}^*(\iota_\theta^* a\iota_\theta)^{-1}\iota_{N_2(\theta)} = a^{\textnormal{hom}}(\theta)^{-1} \quad (\theta \in \Theta),
\end{equation}
with $\Re a^{\textnormal{hom}}(\theta)^{-1}\geq \nu\|a\|^{-2} 1_{N_2(\theta)}$.

Fix $\beta \in \C^{n \times d}$. Since $e^{\i \langle \theta,\cdot\rangle_{\mathbb{C}^d}}\mathbb{C}^{n \times d}\subseteq P(\theta)$,  and $\iota_\theta^* a \iota_\theta : P(\theta) \rightarrow P(\theta)$ is invertible, there exists $\gamma\in \mathbb{C}^{n \times d}$ and $N_{\theta \gamma}\in [H^1_\theta(Y)\perp e^{\i \langle \theta,\cdot\rangle_{\mathbb{C}^d}}]^n$ such that 
		\begin{equation}
		\label{pfdiage6}
		\iota_\theta^* a\iota_\theta(e^{\i \langle \theta,\cdot\rangle_{\mathbb{C}^d}}\gamma+ \grad_\theta N_{\theta \gamma})=e^{\i \langle \theta,\cdot\rangle_{\mathbb{C}^d}}\beta.
		\end{equation}
		Next, we compute for all $q\in [H^1_\theta(Y)\perp e^{\i \langle \theta,\cdot\rangle_{\mathbb{C}^d}}]^n$ that 
		\begin{align*}
		0 & = \langle e^{\i \langle \theta,\cdot\rangle_{\mathbb{C}^d}}\beta, \grad_\theta q \rangle = \big\langle \iota_\theta^* a\iota_\theta(e^{\i \langle \theta,\cdot\rangle_{\mathbb{C}^d}}\gamma+ \grad_\theta N_{\theta \gamma}), \grad_\theta q  \big\rangle \\ 
		&= \big\langle a(\grad_\theta N_{\theta \gamma} + e^{\i \langle \theta,\cdot\rangle_{\mathbb{C}^d}}\gamma), \grad_\theta q  \big\rangle.
		\end{align*}
		That is, $(N_{\theta \gamma k})_{k \in \{1,\ldots,n\}} = \big( \sum_{r=1}^n \sum_{s=1}^d  N^{(rs)}_{\theta k}  \gamma_{rs} \big)_{k \in \{1,\ldots,n\}}$, where $N^{(rs)}_\theta$ uniquely solves \eqref{cellprob}. 
		Furthermore, since $  e^{\i \langle \theta,\cdot\rangle_{\mathbb{C}^d}}\C^{n\times d}\subseteq P(\theta)$, \eqref{pfdiage6} implies that
		\begin{align}
		\notag\langle \beta, \eta \rangle_{\mathbb{C}^{n \times d}}  &= \langle e^{\i \langle \theta,\cdot\rangle_{\mathbb{C}^d}}\beta,e^{\i \langle \theta,\cdot\rangle_{\mathbb{C}^d}}\eta \rangle = \big\langle \iota_\theta^* a\iota_\theta(e^{\i \langle \theta,\cdot\rangle_{\mathbb{C}^d}}\gamma+ \grad_\theta N_{\theta \gamma}), e^{\i \langle \theta,\cdot\rangle_{\mathbb{C}^d}}\eta \big\rangle \\
		&= \big\langle a(\grad_\theta N_{\theta \gamma} + e^{\i \langle \theta,\cdot\rangle_{\mathbb{C}^d}}\gamma), e^{\i \langle \theta,\cdot\rangle_{\mathbb{C}^d}}\eta \big\rangle  = \langle a^{\textnormal{hom}}(\theta) \gamma, \eta \rangle_{\mathbb{C}^{n \times d}}, \qquad (\eta \in \mathbb{C}^{n\times d}).\label{eq:hom1}
		\end{align}
		That is $
		\gamma = \big( a^{\textnormal{hom}}(\theta) \big)^{-1} \beta,$
		where $a^{\textnormal{hom}}(\theta)$ is given by \eqref{ahomtheta}. Hence,
		\begin{align*}
	e^{\i \langle \theta,\cdot\rangle_{\mathbb{C}^d}}	\big( a^{\textnormal{hom}}(\theta) \big)^{-1} \beta & = 	e^{i\langle \theta,\cdot\rangle_{\mathbb{C}^d}}\gamma = \iota_{N_2(\theta)}^* 	e^{\i \langle \theta,\cdot\rangle_{\mathbb{C}^d}}\gamma \\
	&= \iota_{N_2(\theta)}^* (\grad_\theta N_{\theta \gamma} + 	e^{i\langle \theta,\cdot\rangle_{\mathbb{C}^d}}\gamma) 
		=  \iota_{N_2(\theta)}^* (\iota_\theta^* a\iota_\theta)^{-1}(	e^{\i\langle \theta,\cdot\rangle_{\mathbb{C}^d}}\beta) 
		\\&=\iota_{N_2(\theta)}^* (\iota_\theta^* a\iota_\theta)^{-1}\iota_{N_2(\theta)}(	e^{\i \langle \theta,\cdot\rangle_{\mathbb{C}^d}}\beta),
		\end{align*}
		that is, we have shown \eqref{homforma} holds. The claimed properties of $a^{\textnormal{hom}}(\theta)$ in the theorem statement are demonstrated in Proposition \ref{prop:ahom} in Section \ref{sec:ahom}.
\end{proof}
In the proof of Theorem \ref{s:example} we proved the following result about the asymptotic behaviour of the fluxes.
\begin{proposition}
For $F \in [L^2(Y)]^n$,	let
	\[
	u_{\ep,\theta} = \big( -\ep^{-2}\dive_\theta a \grad_\theta + s \big)^{-1} F,
	\]
	and
	\[
	v_{\ep,\theta} = \big(-\ep^{-2} \dive_\theta a^{\rm hom}(\theta) \grad_\theta + m(s) \big)^{-1} F.
	\]
	Then,
	\[
\Vert \ep^{-1}	\pi_\theta  a \grad_\theta u_{\ep \theta} - \ep^{-1}	\pi_\theta a^{\rm hom}(\theta) \grad_\theta v_{\ep,\theta} \Vert_{[L^2(Y)]^{n \times d}} \le \kappa \ep \Vert F \Vert_{[L^2(Y)]^n}.
	\]
\end{proposition}
\begin{proof}
This follows from inequalities \eqref{ffres}, \eqref{eq:homform} and \eqref{ffresss} for right-hand side $\mathcal{U}_\ep^{-1} (F,0)^T$.
\end{proof}
Another implication of Theorem \ref{t:hypasse}, which we use in the next section, is the analogue of Theorem \ref{t:mtgrtheta2} that reads as follows.
\begin{theorem}
	\label{t.ahom2}
Let $a \in \mathcal{M}_{n,d}^\#$, $s \in \mathcal{S}_{n,d}^\#$. Consider the setting \eqref{block:ell} and let $\tilde{\iota}_\theta:[L^2(Y)]^n\oplus P(\theta) \hookrightarrow [L^2(Y)]^n\oplus[L^2(Y)]^{n \times d}$, $\iota\colon [L^2(\mathbb{R}^d)]^n\hookrightarrow [L^2(\mathbb{R}^d)]^n\oplus \{0\}\subseteq [L^2(\mathbb{R}^d)]^n\oplus[L^2(\mathbb{R}^d)]^{n \times d}$. Then,  there exists $\kappa>0$ such that for all $\epsilon>0$, one has
	\begin{multline*}
	\Bigl\| \big( -\dive a\left(\tfrac{\cdot}{\ep}\right) \grad + s\left(\tfrac{\cdot}{\ep}\right) \big)^{-1} -  \iota^*\mathcal{U}_\epsilon^*\int_{\Theta}^\oplus \tilde{\iota}_\theta \iota_{N(\theta)}  \bigg( \pi_{N(\theta)} \bigg[  \begin{pmatrix} \m(s) & 0 \\ 0& a^{\textnormal{hom}}(\theta)^{-1}\end{pmatrix} \\ + \frac{1}{\epsilon} \begin{pmatrix} 0 & -\dive_\theta\iota_\theta \\ -\iota_\theta \grad_\theta & 0\end{pmatrix} \bigg] \pi_{N(\theta)} \bigg)^{-1}\iota_{N(\theta)}^*\tilde{\iota_\theta}^*d\theta \mathcal{U}_\epsilon \iota \Bigr\| \leq \kappa \epsilon.
	\end{multline*}
\end{theorem}
For completeness, we shall end this section with the well-posedness proof of \eqref{cellprob}.
\begin{proposition}\label{p:ahomwd} Let $a\in \mathcal{M}_{n,d}^\#$. Then, for all $\theta\in\Theta$ and $\gamma\in \mathbb{C}^{n \times d}$ there exists a uniquely determined $N_{\theta \gamma}\in [H^1_\theta(Y)\bot e^{\i\langle \theta,\cdot\rangle_{\mathbb{C}^d}}]^n$ such that
	\begin{equation}\label{eq:homvar}
	\langle a (\grad_\theta N_{\theta \gamma} + e^{\i \langle \theta,\cdot\rangle_{\mathbb{C}^d}}\gamma),\grad_\theta\phi\rangle = 0, \quad  (\phi\in [H^1_\theta(Y)\bot e^{\i\langle \theta,\cdot\rangle_{\mathbb{C}^d}}]^n).
	\end{equation} 
	Furthermore, the inequality 
	$$
	\Vert \grad_\theta N_{\theta \gamma} \Vert \le\tfrac{\Vert a \Vert}{\nu} \Vert \gamma \Vert
	$$
	holds. Here, $\nu$ is such that $\Re a \ge \nu$.
\end{proposition}
\begin{proof} 
	For this note that by Proposition \ref{p:hypasse3}\ref{en:h31.5}, we have \[R_2(\theta)\coloneqq \{ \grad_\theta \phi; \phi\in [H_\theta^1(Y) \bot e^{\i \langle \theta,\cdot\rangle_{\mathbb{C}^d}}]^n\}\subseteq L^2(Y)^{n\times d}\text{ closed.}\] We denote, as usual, by $\iota_{R_2(\theta)}$ and $\pi_{R_2(\theta)}$ the canonical embedding from $R_2(\theta)$ and the orthogonal projection to $R_2(\theta)$.
	
	Next, we shall reformulate \eqref{eq:homvar}:  $N_{\theta \gamma} \in \mathcal{H}_\theta\coloneqq [H^1_\theta(Y)\bot e^{\i\langle \theta,\cdot\rangle_{\mathbb{C}^d}}]^n$ satisfies \eqref{eq:homvar}, if, and only if, for all $\phi \in \mathcal{H}_\theta$ one has
	\begin{align*}
	  \langle a \grad_\theta N_{\theta \gamma},\grad_\theta\phi\rangle & = \langle -a e^{\i \langle \theta,\cdot\rangle_{\mathbb{C}^d}}\gamma,\grad_\theta\phi\rangle
	  \\ & = \langle -a e^{\i \langle \theta,\cdot\rangle_{\mathbb{C}^d}}\gamma,\pi_{R_2(\theta)}\grad_\theta\phi\rangle
	  \\ & = \langle -a e^{\i \langle \theta,\cdot\rangle_{\mathbb{C}^d}}\gamma,\iota_{R_2(\theta)}\iota_{R_2(\theta)}^*\grad_\theta\phi\rangle
	  \\ & = \langle -\iota_{R_2(\theta)}^* a e^{\i \langle \theta,\cdot\rangle_{\mathbb{C}^d}}\gamma,\iota_{R_2(\theta)}^*\grad_\theta\phi\rangle.
	\end{align*}
        Next, since
        \begin{align*}
           \langle a \grad_\theta N_{\theta \gamma},\grad_\theta\phi\rangle & =\langle a \iota_{R_2(\theta)}\iota_{R_2(\theta)}^* \grad_\theta N_{\theta \gamma},\iota_{R_2(\theta)}\iota_{R_2(\theta)}^*\grad_\theta\phi\rangle
           \\&=\langle \iota_{R_2(\theta)}^* a \iota_{R_2(\theta)}\iota_{R_2(\theta)}^* \grad_\theta N_{\theta \gamma},\iota_{R_2(\theta)}^*\grad_\theta\phi\rangle,
        \end{align*}
        we deduce that \eqref{eq:homvar} is equivalent to stating that
        \[
           \langle \iota_{R_2(\theta)}^* a \iota_{R_2(\theta)}\iota_{R_2(\theta)}^* \grad_\theta N_{\theta \gamma},\iota_{R_2(\theta)}^*\grad_\theta\phi\rangle = \langle -\iota_{R_2(\theta)}^* a e^{\i \langle \theta,\cdot\rangle_{\mathbb{C}^d}}\gamma,\iota_{R_2(\theta)}^*\grad_\theta\phi\rangle \quad(\phi\in \mathcal{H}_\theta),
        \]which, due to the fact that  the operator $\grad_\theta \colon \mathcal{H}_\theta\to R_2(\theta)$ is a bijection, is  equivalent to stating
        \[
           \iota_{R_2(\theta)}^* a \iota_{R_2(\theta)}\iota_{R_2(\theta)}^* \grad_\theta N_{\theta \gamma} = -\iota_{R_2(\theta)}^* a e^{\i \langle \theta,\cdot\rangle_{\mathbb{C}^d}}\gamma.
        \]
        The coerciveness of $a$ implies that $\iota_{R_2(\theta)}^* a \iota_{R_2(\theta)}$ is coercive. Hence,
        \[
          \iota_{R_2(\theta)}^*\grad_\theta N_{\theta \gamma} = -\left(\iota_{R_2(\theta)}^* a \iota_{R_2(\theta)}\right)^{-1}\iota_{R_2(\theta)}^* a e^{\i \langle \theta,\cdot\rangle_{\mathbb{C}^d}}\gamma.
        \]
        The last equation determines $\grad_\theta N_{\theta \gamma}$ uniquely, and the desired assertion follows by observing that $N_{\theta \gamma}\in \mathcal{H}_\theta$ and that $\grad_\theta\colon \mathcal{H}_\theta\to R_2(\theta)$ is bijective.
        
        To prove the inequality, we note that since $\Re a \ge \nu 1_{\mathbb{C}^{n\times d}}$ then we obtain $\Re (\iota^*_\theta a \iota_\theta) \ge \nu 1_{P(\theta)}$. Therefore,
        \[
        \Vert (\iota^*_\theta a \iota_\theta)^{-1} \Vert \le \nu^{-1},
        \]
        and we calculate
        \[
        \Vert \iota_{R_2(\theta)}^*\grad_\theta N_{\theta \gamma} \Vert \le \nu^{-1} \Vert \iota_{R_2(\theta)}^* a e^{\i \langle \theta,\cdot\rangle_{\mathbb{C}^d}}\gamma \Vert \le \tfrac{\Vert a \Vert}{\nu} \Vert \gamma \Vert. \qedhere
        \]
\end{proof}

\section{Properties of the fibre-homogenised matrix $a^{\textnormal{hom}}(\theta)$ and comparisons to classical results}
\label{sec:ahom}

In the whole section, we adopt the setting \eqref{block:ell}. In Section \ref{s:example}, we established 
$$
 \mathcal{U}_\epsilon^{-1}\int_{\Theta}^\oplus \big(- \ep^{-2}\dive_\theta a^{\textnormal{hom}}(\theta) \grad_\theta  + \m(s)\big)^{-1}d\theta\mathcal{U}_\epsilon$$
to be non-standard leading-order asymptotics in $\ep>0$, uniform in $\theta \in \Theta$, for the operator family $\big((- \dive a(\cdot/\epsilon)  \grad  + s)^{-1}\big)_\epsilon$. 
This section is devoted to comparing these asymptotics to the classical ones found in the literature, see Remark \ref{r.compare}. We end the section with an example of when $A^{\rm hom}(\theta) \neq A^{\rm hom}(0)$. The main result of the section is as follows. 
\begin{theorem}
	\label{t.classical}
 Let $a \in \mathcal{M}_{n,d}^\#$, $s \in \mathcal{S}_{n,d}^\#$. Then, there exists a constant $\kappa>0$ such that for all $\epsilon>0$, the inequality
 \[
 \Bigl\| \big(-\dive a\left(\tfrac{\cdot}{\epsilon}\right)\grad + s\left(\tfrac{\cdot}{\epsilon}\right)\big)^{-1}  - \big(-\dive a^{\textnormal{hom}}(0)\grad + \m(s)\big)^{-1} \Bigr\| \leq \kappa \epsilon
 \]
 holds.  The constant matrix $\m(s) \in \mathcal{S}_{n,d}^\#$ and constant fourth-order tensor $a^{\textnormal{hom}}(0) \in \mathcal{M}_{n,d}^\#$, are given in Theorem \ref{t:quanthom}.
\end{theorem}
Before proving this result, we introduce some related auxiliary results. 
\begin{proposition}
	\label{prop:ahom}
	Let  $\theta\in \Theta$, $a \in \mathcal{M}_{n,d}^\#$ and $\nu>0$ such that $\Re a\geq \nu 1_{\C^{n \times d}}$. Then, the following assertions hold:
	\begin{enumerate}[label=(\alph*)]
	 \item\label{ahom1} for all $X\in\mathbb{C}^{n\times d}$ with $N_{\theta X}\coloneqq \sum_{r=1}^n\sum_{s=1}^d N_{\theta}^{(rs)}X_{rs}$
	 \[
	     \langle a^{\textnormal{hom}}(\theta)X,Z\rangle_{\mathbb{C}^{n\times d}}=\langle a(\grad_\theta N_{\theta X} +e^{\i\langle\theta,\cdot\rangle_{\C^d}}X),\grad_\theta N_{\theta Z} +e^{\i\langle\theta,\cdot\rangle_{\C^d}}Z\rangle \quad (X,Z \in \mathbb{C}^{n \times d});
	 \]
	 \item\label{ahom2} for all $X\in\mathbb{C}^{n\times d}$
	 \begin{multline*}
	     \Re\langle a^{\textnormal{hom}}(\theta)X,X\rangle_{\mathbb{C}^{n\times d}}\\ =\inf_{N\in [H_\theta^1(Y)\bot e^{\i\langle \theta,\cdot\rangle_{\C^d}}]^n}\Re \langle a(\grad_\theta N_\theta +e^{\i\langle\theta,\cdot\rangle_{\C^d}}X),\grad_\theta N_\theta +e^{\i\langle\theta,\cdot\rangle_{\C^d}}X\rangle;
	 \end{multline*}
	 \item\label{ahom2.5} we have
	 \[
	     \iota_{N_2(\theta)}^*\left(\iota_\theta^*a\iota_\theta\right)^{-1}\iota_{N_2(\theta)}=a^{\textnormal{hom}}(\theta)^{-1};
	 \]
	 \item\label{ahom3} $\Re a^{\textnormal{hom}}(\theta)\geq \nu 1_{\C^{n\times d}}$;
	 \item\label{ahom4} $\|\Re a^{\textnormal{hom}}(\theta)\|\leq \|\Re a\|$;
	 \item\label{ahom4.5} $\|a^{\textnormal{hom}}(\theta)\|\leq \frac{\|a\|^2}{\nu}$;
	 \item\label{ahom5} if $a_{ijkl}\in \mathbb{C}$ $(i,k \in \{1, \ldots, n\} , j,l\in \{ 1,\ldots, d\})$, then $a^{\textnormal{hom}}(\theta) = a$ $(\theta \in \Theta)$.
	\end{enumerate}
\end{proposition}
\begin{proof}
To prove \ref{ahom1}, we use \eqref{eq:hom1} and observe that 
\[
    \langle a(\grad_\theta N_{\theta X} +e^{\i\langle\theta,\cdot\rangle_{\C^d}}X),\grad_\theta N_{\theta Z}\rangle=0
\]as $N_{\theta X}\in [H_\theta^1(Y)\bot e^{\i\langle\theta,\cdot\rangle_{\C^d}}]^n$.
Next, the claim in \ref{ahom2} follows from the observation that \eqref{cellprob} is the Euler--Lagrange equation corresponding to the problem of finding the minimiser of the non-negative functional
\[
      [H_\theta^1(Y)\bot e^{\i\langle \theta,\cdot\rangle_{\C^d}}]^n\ni N\mapsto \Re \langle a(\grad_\theta N_\theta +e^{\i\langle\theta,\cdot\rangle_{\C^d}}X),\grad_\theta N_\theta +e^{\i\langle\theta,\cdot\rangle_{\C^d}}X\rangle.
  \]
  The assertion \ref{ahom2.5} is shown in \eqref{homforma}.
  For the proof of \ref{ahom3}, we let $X\in \mathbb{C}^{n\times d}$ and use \ref{ahom1} to obtain
  \begin{align*}
     \Re\langle a^{\textnormal{hom}}(\theta)X,X\rangle_{\mathbb{C}^{n\times d}}& =\Re\langle a(\grad_\theta N_{\theta X} +e^{\i\langle\theta,\cdot\rangle_{\C^d}}X),\grad_\theta N_{\theta X} +e^{\i\langle\theta,\cdot\rangle_{\C^d}}X\rangle
     \\ &\geq \nu \langle \grad_\theta N_{\theta X} +e^{\i\langle\theta,\cdot\rangle_{\C^d}}X,\grad_\theta N_{\theta X} +e^{\i\langle\theta,\cdot\rangle_{\C^d}}X\rangle
     \\ &= \nu (\|\grad_\theta N_{\theta X}\|^2+\|X\|^2)
     \\ &\geq \nu\|X\|^2,
  \end{align*}
 where we used Pythagoras' identity as $\grad_\theta N_{\theta X}\bot e^{\i\langle\theta,\cdot\rangle_{\C^d}}X$.
 
 In order to prove \ref{ahom4}, we shall use \ref{ahom2}. Indeed, for all $X\in \mathbb{C}^{n\times d}$, we obtain
 \begin{align*}
    \Re\langle a^{\textnormal{hom}}(\theta)X,X\rangle_{\mathbb{C}^{n\times d}}&=\inf_{N\in [H_\theta^1(Y)\bot e^{\i\langle \theta,\cdot\rangle_{\C^d}}]^n} \Re\langle a(\grad_\theta N_\theta +e^{\i\langle\theta,\cdot\rangle_{\C^d}}X),\grad_\theta N_\theta +e^{\i\langle\theta,\cdot\rangle_{\C^d}}X\rangle 
    \\&\leq \Re    \langle a e^{\i\langle\theta,\cdot\rangle_{\C^d}}X,e^{\i\langle\theta,\cdot\rangle_{\C^d}}X\rangle \leq \|\Re a\|\|X\|^2
 \end{align*}
 The proof of \ref{ahom4.5} uses \ref{ahom2.5}. From the inequality $\Re a\geq \nu 1_{\C^{n \times d}}$, we infer that $\Re \iota_\theta^* a \iota_\theta\geq \nu 1_{\C^{n \times d}}$. Hence, $\Re (\iota_\theta^*a\iota_\theta)^{-1}\geq \nu/(\|a\|^2) 1_{\C^{n \times d}}$. Thus, 
 \[
     \|a^{\textnormal{hom}}(\theta)\|=\left\|\left(a^{\textnormal{hom}}(\theta)^{-1}\right)^{-1}\right\|\leq \tfrac{\|a\|^2}{\nu}.
 \]
        The last assertion follows from the observation that a constant $a$ leaves $N_2(\theta)$ and, hence, $P(\theta)$ invariant. Therefore, we obtain
        \begin{align*}
              a^{\textnormal{hom}}(\theta)^{-1}& = \iota_{N_2(\theta)}^*(\iota_\theta^*a\iota_\theta)^{-1}\iota_{N_2(\theta)}
              \\   & = \iota_{N_2(\theta)}^*\iota_\theta^*\iota_\theta a^{-1}\iota_{N_2(\theta)}
              \\   & = \iota_{N_2(\theta)}^*\iota_\theta^*\iota_\theta \iota_{N_2(\theta)}  a^{-1}
              \\   & = a^{-1}\qedhere
        \end{align*}
\end{proof}
\begin{proposition}
	\label{prop:ahomasym}
	There exists $\kappa>0$ such that for all $\theta \in \Theta$
	\[
	\Vert a^{\textnormal{hom}}(\theta) - a^{\textnormal{hom}}(0) \Vert_{N_2(\theta)}	\le \kappa|\theta|.
	\]
\end{proposition}
\begin{proof}
	As $N^{(rs)}_\theta$ solves \eqref{cellprob}, then Proposition \ref{p:hypasse3} \ref{en:h31} and Proposition \ref{p:ahomwd} imply that
	\begin{equation}
	\label{cellbound}
	\| N^{(rs)}_\theta \|_{[H^1(Y)]^n} \le \big( \pi^{-2} + 1 \big)^{1/2} \tfrac{\Vert a \Vert}{\nu}, \qquad ( \theta \in \Theta, \, r\in\{1,\ldots,n\}, s \in \{1,\ldots,d\}).
	\end{equation}
Using the notation in Proposition \ref{prop:ahom}, assertion \eqref{eq:hom1} implies  that
\begin{equation*}
\langle a^{\textnormal{hom}}(\theta)X,Z\rangle_{N_2(\theta)}=\langle a(\grad_\theta N_{\theta X} +e^{\i\langle\theta,\cdot\rangle_{\C^{d}}}X), e^{\i\langle\theta,\cdot\rangle_{\C^{d}}}Z\rangle \quad (X,Z \in \mathbb{C}^{n \times d});
\end{equation*}
This identity   yields
\begin{multline*}
	\langle	\left(a^{\textnormal{hom}}(\theta) - a^{\textnormal{hom}}(0) \right) X, Z \rangle_{N_2(\theta)}  = \langle a( \grad_\theta  N_{\theta X} - \grad_{\#}N_{0 X}), e^{\i\langle\theta,\cdot\rangle_{\C^{d}}}Z\rangle \\
	+   \langle a\grad_{\#}N_{0 X}, (e^{\i\langle\theta,\cdot\rangle_{\C^{d}}}-1)Z\rangle\quad (X,Z \in \mathbb{C}^{n \times d}).
\end{multline*}
 Consequently 
	\[
	\|a^{\textnormal{hom}}(\theta) - a^{\textnormal{hom}}(0) \| \leq \Vert a \Vert \Vert \grad_\theta  N_{\theta X} - \grad_{\#}N_{0 X}\Vert  + | \theta | \Vert a \grad_{\#}N_{0 X} \Vert.
	\]
Recalling \eqref{cellbound}, we observe that to prove the proposition it remains to demonstrate
	\begin{equation}
	\label{ahomasswant}
\exists \kappa>0 \, \forall \theta \in \Theta \quad \Vert \grad_\theta  N_{\theta X} - \grad_{\#}N_{0 X} \Vert
 \le \kappa | \theta| \Vert X \Vert.
	\end{equation}
By \eqref{cellprob}, one has for $N_{\theta X} =   \sum_{r=1}^n \sum_{s=1}^d  N^{(rs)}_\theta  X_{rs}$ that 
	\begin{align*}
	\langle a  \grad_\theta N_{\theta X}, \grad_\theta \phi \rangle = - \langle  a e^{\i \langle \theta, \cdot \rangle_{\mathbb{C}^d}} X ,  \grad_\theta \phi\rangle, \qquad (\phi \in [H^1_\theta(Y) \perp e^{\i \langle \theta, \cdot \rangle_{\mathbb{C}^d}}]^n),
	\end{align*}
	and
	\begin{align*}
	\langle a  \grad_\# N_{0 X}, \grad_\# \phi_0 \rangle = - \langle  a X,  \grad_\# \phi_0\rangle, \qquad (\phi_0 \in [H^1_\#(Y) \perp 1]^n).
	\end{align*}
Fix $\phi_0 \in [H^1_\#(Y) \perp 1]^n$, and set $\widetilde{N}_{ \theta X} = e^{-\i \langle \theta, \cdot \rangle_{\mathbb{C}^d}} N_{\theta X}$, $\phi =e^{\i \langle \theta, \cdot \rangle_{\mathbb{C}^d}} \phi_0$. Clearly $\widetilde{N}_{\theta X}$ belongs to  $[H^1_\#(Y) \perp 1]^n$ and $\phi \in[H^1_\theta(Y) \perp  e^{\i \langle \theta, \cdot \rangle_{\mathbb{C}^d}}]^n $. By the identity $\grad_\theta  e^{\i \langle \theta, \cdot \rangle_{\mathbb{C}^d}}  =  e^{\i \langle \theta, \cdot \rangle_{\mathbb{C}^d}} (\grad_\# + \i \theta) $, and the equation for $N_{\theta X}$, we calculate that $\widetilde{N}_{ \theta X}$  solves 
	\begin{align*}
	\langle a  \grad_\# \widetilde{N}_{ \theta X}, \grad_\# \phi_0 \rangle = - \langle  a  \grad_\# \widetilde{N}_{ \theta X}, \i \theta  \phi_0\rangle- \langle  a  ( \i \theta \widetilde{N}_{ \theta X} + X),  (\grad_\# + \i \theta)\phi_0\rangle.
	\end{align*}
	 Therefore, 
	\begin{equation}
	\label{ahomasse1}
	\langle  a  \grad_\# [ \widetilde{N}_{ \theta X}- N_{0 X}]  ,  \grad_\# \phi_0 \rangle = R_\theta(\phi_0),
	\end{equation}
	where
	\begin{align*}
	R_\theta(\phi_0) \coloneqq  - \langle  a  \grad_\# \widetilde{N}_{ \theta X}, \i \theta  \phi_0 \rangle- \langle  a   \i \theta \widetilde{N}_{ \theta X},  (\grad_\# + \i \theta)\phi_0 \rangle - \langle  a  X ,\i \theta \phi_0 \rangle.
	\end{align*}
	Utilising \eqref{cellbound} and Propostion \ref{p:hypasse3} \ref{en:h31} gives
	\[
	| R_\theta(\phi_0) | \le \kappa | \theta | \Vert X \Vert \| \phi_0 \|_{[H^1(Y)]^n} \le \kappa (1+ \pi^{-2})^{1/2} | \theta | \Vert X \Vert  \| \grad_{\#}\phi_0 \|_{[L^2(Y)]^{n \times d}} ,
	\]
	By setting $\phi_0 =  \widetilde{N}_{ \theta X} - N_{0 X} $, and recalling that $\Re a \ge \nu 1_{\C^{n \times d}}$ gives the inequality  \eqref{ahomasswant}. Hence, the proposition is proved.
\end{proof}

The last step in proving Theorem \ref{t.classical} is contained in the next proposition.

\begin{proposition}
	\label{p.ahomasy}
	There exists a constant $\kappa>0$ such that, for all $ \theta \in \Theta$, $\epsilon>0$, and $f\in \mathbb{C}^n$, $f_\theta\coloneqq e^{\i\langle \theta,\cdot\rangle_{\mathbb{C}^d}} f$ with
	\begin{align*}
	  \begin{pmatrix}
	   \beta_\theta
	   \\ M_\theta
	  \end{pmatrix}&\coloneqq \left( \begin{pmatrix}
	\m(s)  & 0 \\ 0& a^{\textnormal{hom}}(\theta)^{-1} 
	\end{pmatrix}   - \tfrac{1}{\ep}  \iota^*_{N(\theta)} \begin{pmatrix}
	0 &  - \dive_\theta \iota_\theta  \\ - \iota^*_\theta \grad_\theta  & 0 
	\end{pmatrix} \iota_{N(\theta)}   \right)^{-1}
	\begin{pmatrix}
	   f_\theta
	   \\ 0
	  \end{pmatrix},
	  \\
	  \begin{pmatrix}
	   \beta_\theta'
	   \\ M_\theta'
	  \end{pmatrix}&\coloneqq \left( \begin{pmatrix}
	\m(s)  & 0 \\ 0& a^{\textnormal{hom}}(0)^{-1} 
	\end{pmatrix}   - \tfrac{1}{\ep}  \iota^*_{N(\theta)} \begin{pmatrix}
	0 &  - \dive_\theta \iota_\theta  \\ - \iota^*_\theta \grad_\theta  & 0 
	\end{pmatrix} \iota_{N(\theta)}   \right)^{-1}
	\begin{pmatrix}
	    f_\theta
	   \\ 0
	  \end{pmatrix}	  
	\end{align*}
	 the following inequality
	\begin{equation*}
	\|\beta_\theta-\beta_\theta'\| \le \kappa \ep \|f_\theta\|
	\end{equation*}
	holds.
\end{proposition}
\begin{proof} Fix $\theta \in \Theta$. Recall that 
	\begin{equation*}
N(\theta) = N_1(\theta) \oplus N_2(\theta), \quad	N_1(\theta)\coloneqq e^{\i \langle\theta, \cdot \rangle_{\mathbb{C}^d}} \mathbb{C}^n,\text{ and }N_2(\theta)\coloneqq e^{\i \langle\theta, \cdot \rangle_{\mathbb{C}^d}}\mathbb{C}^{n \times d}. \end{equation*}
By direct calculation, it follows that $ \rge (\dive_\theta \vert_{N_2(\theta)} ) \subseteq N_1(\theta)$, $\rge ( \grad_\theta \vert_{N_1(\theta)} ) \subseteq N_2(\theta)$ and that
$$
\iota^*_{N(\theta)} \begin{pmatrix}
0 &  - \dive_\theta \iota_\theta  \\ - \iota^*_\theta \grad_\theta  & 0 
\end{pmatrix} \iota_{N(\theta)} \left( \begin{matrix}
e^{\i \langle\theta, \cdot \rangle_{\mathbb{C}^d}} \beta  \\ e^{\i \langle\theta, \cdot \rangle_{\mathbb{C}^d}} M
\end{matrix}  \right) = -\left( \begin{matrix}
e^{\i \langle\theta, \cdot \rangle_{\mathbb{C}^d}} M \i \theta  \\ e^{\i \langle\theta, \cdot \rangle_{\mathbb{C}^d}} \beta \otimes \i \theta 
\end{matrix}  \right), \quad (\beta \in\mathbb{C}^d, M \in \mathbb{C}^{n \times d} ).
$$

Let us now prove the desired assertion. For $f \in \mathbb{C}^n$, consider the problem: Find $(\beta,M) \in \C^n \oplus \C^{n \times d}=N(0)$ such that
\[
 \begin{pmatrix}
\m(s)  & 0 \\ 0& a^{\textnormal{hom}}(\theta)^{-1} 
\end{pmatrix}   \begin{pmatrix}
e^{\i \langle\theta, \cdot \rangle_{\mathbb{C}^d}}\beta \\ e^{\i \langle\theta, \cdot \rangle_{\mathbb{C}^d}} M 
\end{pmatrix}  - \tfrac{1}{\ep}  \begin{pmatrix}
e^{\i \langle\theta, \cdot \rangle_{\mathbb{C}^d}} M \i \theta  \\ e^{\i \langle\theta, \cdot \rangle_{\mathbb{C}^d}} \beta \otimes \i \theta 
\end{pmatrix}   = \begin{pmatrix}
e^{\i \langle\theta, \cdot \rangle_{\mathbb{C}^d}}f \\ 0
\end{pmatrix},
\]
equivalently $
M = \tfrac{1}{\ep} a^{\textnormal{hom}}(\theta)  \beta \otimes \i \theta,
$ and
\begin{equation*}
 \ep^{-2} a^{\textnormal{hom}}(\theta) (  \beta \otimes \theta)     \theta + \m(s)  \beta = f.
\end{equation*}
By taking the inner product on both sides of the above identity with $\beta$, we calculate
\begin{equation}
\label{limitbd}
\big( \ep^{-2} \nu |   \theta |^2 + \nu \big) \Vert \beta \Vert\le \Vert f \Vert,
\end{equation}
where $\nu>0$ is such that for all $\theta \in \Theta$ we have $\Re a^{\textnormal{hom}}(\theta) \ge \nu 1_{\C^{n \times d}}$, and $\Re \m(s) \ge \nu 1_{\mathbb{C}^{n}}$; note that such $\nu$ exists by Proposition \ref{prop:ahom}.

Now, let $\beta' \in \C^n$ solve 
\begin{equation*}
\ep^{-2} a^{\textnormal{hom}}(0) ( \beta' \otimes  \theta)   \theta + \m(s)  \beta' = f.
\end{equation*}
It follows that 
\begin{equation*}
\ep^{-2}a^{\textnormal{hom}}(0)  (( \beta - \beta') \otimes \theta ) \theta + \m(s) ( \beta - \beta') =  \ep^{-2}  ( a^{\textnormal{hom}}(0) - a^{\textnormal{hom}}(\theta)) (\beta \otimes \theta  ) \theta .
\end{equation*}
Therefore, arguing as in the derivation of inequality \eqref{limitbd} with $a^{\textnormal{hom}}(0)$ instead of $a^{\textnormal{hom}}(\theta)$ and $f=\ep^{-2}  ( a^{\textnormal{hom}}(0) - a^{\textnormal{hom}}(\theta)) (\beta \otimes \theta  ) \theta$, we deduce that
\begin{flalign*}
\big( \ep^{-2} \nu |\theta |^2 + \nu \big) \Vert \beta - \beta' \Vert & \le \Vert  \ep^{-2}  [ a^{\textnormal{hom}}(0) - a^{\textnormal{hom}}(\theta)]    (\beta \otimes \theta ) \theta \Vert \\
& \le \ep^{-2}\Vert a^{\textnormal{hom}}(0) - a^{\textnormal{hom}}(\theta) \Vert | \theta |^2 \Vert \beta \Vert. 
\end{flalign*}
Consequently, by considering  Proposition \ref{prop:ahomasym} and \eqref{limitbd} for arbitrary $f\in\mathbb{C}^n$ again,  we deduce that
$$
\Vert \beta - \beta' \Vert \le \kappa \tfrac{\ep^{-2} | \theta |^3}{\nu^2 \big( \ep^{-2} |   \theta |^2 + 1 \big)^2} \Vert f \Vert \le C \ep  \Vert f \Vert
$$
for $C = \kappa \sup_{t \in \R} \tfrac{t^3}{\nu^2(1+t^2)^2}$. Hence, Proposition \ref{p.ahomasy} holds.
\end{proof}
\begin{proof}[Proof of Theorem \ref{t.classical}]
In Proposition \ref{prop:ahom}  we established that $a^{\textnormal{hom}}(\theta) \in \mathcal{M}_{n,d}^{\#}$ and that if $a \in \mathcal{M}_{n,d}^{\#}$ is constant then  $a^{\textnormal{hom}}(\theta) = a$ for all $\theta \in \Theta$. Furthermore, if $s \in \mathcal{S}_{n,d}^{\#}$ is constant then it clearly follows that $\m(s) =s $. Thus, by utilising Theorem \ref{t.ahom2} twice, once for $a$ and $s$, and again for $a^{\textnormal{hom}}(0)$ and $\m(s)$, we conclude that Theorem \ref{t.classical} follows from Proposition \ref{p.ahomasy}.
\end{proof}

\subsection*{An example of when $a^{\rm hom}(\theta) \neq a^{\rm hom}(0)$.}
Let us recall the well-known result  that if $a \in \mathcal{M}^{\#}_{n,d}$ is self-adjoint and satisfies the assumption $a X \in \ker(\dive_{\#})$ for all $X\in\mathbb{C}^{n\times d}$ then
\[
a^{\rm hom}(0) = a^{\rm hom} = \langle a \rangle\coloneqq \int_Y a(y)dy.
\]
For the reader's convenience we shall reprove this result here (for further information see for example \cite[Section 1.6]{JKO}). The claimed identity can be immediately seen by noting that, for such an $a$, problem \eqref{cellprob} takes the form: Find $N_{0}^{(rs)}\in [H^1_{\#}(Y) \perp 1]^n$ such that
\[
\langle a \grad_{\#} N_{0}^{(rs)}, \grad_{\#} \phi \rangle =0, \qquad (\phi \in [H^1_{\#}(Y) \perp 1]^n),
\]Indeed, this follows from
\[
\langle a e_r\otimes e_s , \grad_{\#} \phi \rangle =-\langle \dive_{\#} a e_r\otimes e_s ,  \phi \rangle=0, \qquad (\phi \in [H^1_{\#}(Y) \perp 1]^n),
\]
Consequently, $N_0^{(rs)}=0$ and from \eqref{ahomtheta} we deduce that $a^{\rm hom}(0) = \langle a \rangle$.

We shall use this observation to demonstrate that in general $a^{\rm hom}(\theta) \neq a^{\rm hom}(0)$ for $\theta \neq 0$. Indeed, the following result holds.
\begin{proposition}
\label{p:ahomdifferent}
Assume  $a \in \mathcal{M}^{\#}_{n,d}$ is self-adjoint with $ aX \in \ker(\dive_{\#})$ for all $X\in\mathbb{C}^{n\times d}$. Then,
\[
a^{\rm hom}(\theta) = a^{\rm hom}(0) \quad (\theta \neq 0)
\]
if, and only if, $a$ is constant.
\end{proposition}
\begin{remark}
For the case $n=d=1$, then the condition $a X = a\cdot X\in \ker(\dive_{\#})$ for all $X\in\mathbb{C}$ (i.e.~$a\in\ker(\dive_\#)$) automatically implies that $a$ is constant. In fact, for the one-dimensional scalar case one does not require the assumption $a \in \ker(\dive_{\#})$ to deduce that $a^{\textnormal{hom}}(\theta)=a^{\textnormal{hom}}(0)$. That is, for any $a \in \mathcal{M}^{\#}_{1,1}$, one can show by direct calculation that $a^{\rm hom}(\theta) = a^{\rm hom}(0)$ for all $\theta \in \Theta$.
\end{remark}
\begin{proof}[Proof of Proposition \ref{p:ahomdifferent}]
Fix $\theta \in \Theta \backslash \{ 0 \}$, $X \in \C^d$. Let $ N_{\theta X} \in [H^1_{\theta}(Y) \perp e^{{\i\langle\theta,\cdot\rangle_{\C^d}}}]^n$ solve

\begin{equation}
\label{pahomdife0}
\langle a(\grad_\theta N_{\theta X} +e^{\i\langle\theta,\cdot\rangle_{\C^d}}X),\grad_\theta \phi \rangle=0, \qquad (\phi \in [H^1_{\theta}(Y) \perp e^{{\i\langle\theta,\cdot\rangle_{\C^d}}}]^n).
\end{equation}
Recalling Proposition \ref{prop:ahom} \ref{ahom1} we deduce that
\begin{equation}
\label{pahomdife1}
\begin{aligned}
\langle a^{\rm hom}(\theta) X, X \rangle_{\C^{n \times d}} & = \langle a\grad_\theta N_{\theta X}, e^{\i\langle\theta,\cdot\rangle_{\C^d}}X\rangle + \langle a e^{\i\langle\theta,\cdot\rangle_{\C^d}}X,e^{\i\langle\theta,\cdot\rangle_{\C^d}}X\rangle \\
& = \langle a\grad_\theta N_{\theta X}, e^{\i\langle\theta,\cdot\rangle_{\C^d}}X\rangle + \langle a \rangle X \cdot X \\
& = \langle a\grad_\theta N_{\theta X}, e^{\i\langle\theta,\cdot\rangle_{\C^d}}X\rangle + \langle a^{\rm hom}(0) X , X\rangle.
\end{aligned}
\end{equation}
Therefore, the identity $a^{\rm hom}(\theta) = a^{\rm hom}(0)$  holds if, and only if, $\langle a\grad_\theta N_{\theta X}, e^{\i\langle\theta,\cdot\rangle_{\C^d}}X\rangle=0$. Note, from the assumptions $a = a^*$ and $aX \in \ker(\dive_{\#})$ for all $X\in\mathbb{C}^{n\times d}$, and the identity 
\[
e^{\i \langle \theta , \cdot \rangle_{\C^d}} \grad_{\#} e^{ -\i \langle \theta , \cdot \rangle_{\C^d}} = \grad_{\theta} + \i \theta,
\] we deduce that
\begin{equation}
\label{pahomdife0.1}
\langle a\grad_\theta \phi, e^{\i\langle\theta,\cdot\rangle_{\C^d}}X\rangle =  - \langle a \i \theta \phi, e^{\i\langle\theta,\cdot\rangle_{\C^d}}X\rangle, \quad (\phi \in [H^1_{\theta}(Y)\perp e^{\i\langle\theta,\cdot\rangle_{\C^d}}]^n).
\end{equation}
Then, from the above assertion it follows that
\[
\langle a\grad_\theta N_{\theta X}, e^{\i\langle\theta,\cdot\rangle_{\C^d}}X\rangle = - \langle a \i \theta N_{\theta X}, e^{\i\langle\theta,\cdot\rangle_{\C^d}}X\rangle.
\]
If we assume $a$ is constant, then the term on the right-hand side of the above equation vanishes  because $N_{\theta X} \perp e^{\i \langle \theta, \cdot \rangle_{\C^d}} \C^d$. Therefore, $a^{\rm hom}(\theta) = a^{\rm hom}(0)$.

Let us assume that $a^{\rm hom}(\theta) = a^{\rm hom}(0)$. We shall now prove that $a$ must necessarily be constant. By \eqref{pahomdife1} it follows that 
\begin{equation}
\label{pahomdife2}
\langle a\grad_\theta N_{\theta X}, e^{\i\langle\theta,\cdot\rangle_{\C^d}}X\rangle=0.
\end{equation}
Equation \eqref{pahomdife2}, the fact $a^* = a$ and setting $\phi = N_{\theta X}$ in \eqref{pahomdife0}, gives 
\[
\langle a\grad_\theta N_{\theta X}, \grad_\theta N_{\theta X}\rangle = 0.
\]
That is, $N_{\theta X} = 0$ and \eqref{pahomdife0} takes the form
\[
\langle a e^{\i\langle\theta,\cdot\rangle_{\C^d}}X,\grad_\theta \phi \rangle=0, \qquad (\phi \in [H^1_{\theta}(Y) \perp e^{{\i\langle\theta,\cdot\rangle_{\C^d}}}]^n).
\]
This in turn, combined with \eqref{pahomdife0.1} and the fact $a=a^*$ implies that
\[
 \langle   a e^{\i\langle\theta,\cdot\rangle_{\C^d}}X, \i \theta \phi\rangle = 0, \quad (\phi \in [H^1_{\theta}(Y)\perp e^{{\i\langle\theta,\cdot\rangle_{\C^d}}}]^n).
\]
That is,
\[
a e^{\i\langle\theta,\cdot\rangle_{\C^d}}X \i \theta  \in e^{\i\langle\theta,\cdot\rangle_{\C^d}} \C^d
\]
which can only be true if $a$ is constant.
\end{proof}
\begin{example}
 We give a small concrete example that the set of $a$ satisfying the conditions imposed in Proposition \ref{p:ahomdifferent} is non-trivial. For this, let $n=1$, $d=2$. Take $\phi,\psi\in C_c^\infty(0,1;\mathbb{R})$ with $\phi, \psi \ge 0$. Then define 
 \[
    b\colon Y\ni (y_1,y_2)\mapsto \begin{pmatrix}
                           \phi(y_2) & 0
                        \\ 0 & \psi(y_1)
                        \end{pmatrix}. \]
As the entries of $b$ are non-negative, $a\coloneqq b+ 1_{\mathbb{C}^3} \in\mathcal{M}_{1,2}^\#$. Moreover, $a=a^*$. The divergence condition, that is both of the columns of $a$ are in the kernel of $\dive_\#$, is easy to see.
\end{example}

\section{Application of fibre homogenisation to equations of Maxwell type}
\label{s:max}
In this section, we shall demonstrate the utility of our approach in the context of Maxwell's equations. That is to say, we shall treat the following static variant of Maxwell's equations:
\[
    \left(\begin{pmatrix}
       \epsilon\big( \tfrac{\cdot}{\eta} \big) & 0 \\ 0 & \mu\big( \tfrac{\cdot}{\eta} \big)
    \end{pmatrix} + \begin{pmatrix} 0 & -\curl \\  \curl&0 \end{pmatrix}\right)\begin{pmatrix} E^\eta \\ H^\eta \end{pmatrix} = \begin{pmatrix} J \\ 0\end{pmatrix} \quad \text{on $\mathbb{R}^3$.}
\]
For consistency with notation in the literature, where $\ep$ is often reserved for the dielectric permittivity, we denote $\eta \in (0,\infty)$ to be the  parameter. Here, $J$, $\ep$, $\mu$ are given and the unknowns $E^\eta,$ and $H^\eta$ are the electric and magnetic fields respectively. A system of the type may occur, for example, when considering the resolvent problem for the  Maxwell system  in the frequency domain at a fixed frequency. The operator $\curl$ is acting as 
\[
    \curl (E_j)_{j\in\{1,2,3\}}\coloneqq \begin{pmatrix} \partial_2 E_3 - \partial_3 E_2 \\ \partial_3 E_1 -\partial_1 E_3 \\ \partial_1 E_2-\partial_2 E_1\end{pmatrix}
\]
realised as an operator in $[L^2(\mathbb{R}^3)]^3$. Note that $\curl$, thus defined, is selfadjoint.

Henceforth, we consider $\eps,\mu\in \mathcal{M}_{1,3}^\#$, that is  we assume that 
\[
\epsilon,\mu\in L^\infty(\mathbb{R}^3;L(\mathbb{C}^3))
\]
are $Y$-periodic and satisfy $\Re\eps(x),\Re\mu(x)\geq \nu$ for some $\nu>0$ and a.e.\ $x\in\mathbb{R}^3$. As the operator $\curl$ is selfadjoint, then by Lemma \ref{l:invblty}, we deduce that for a given $J \in [L^2(\mathbb{R}^3)]^3$ there exists a unique pair $(E^\eta, H^\eta) \in [\dom(\curl_\theta)]^2$ to the above Maxwell system. The rest of the section focuses on describing the small $\eta$ behaviour of this solution via the approach described in Section \ref{s:gen}.

Let $\mathcal{U}_\eta$ be the Gelfand transform introduced in Section \ref{s:example},  Definition \ref{def.gt}. The following result states that $U_\eta$ interacts with $\curl$ in a similar way to its interaction with $\grad$ and $\dive$. 
\begin{proposition}\label{p:cutim}
  For all $\eta>0$, we have
  \[
     \mathcal{U}_\eta \curl \mathcal{U}_\eta^{-1} =  \int_{[-\pi,\pi)^3}^\oplus \tfrac{1}{\eta}\curl_\theta d\theta,
  \]
  where $\curl_\theta \coloneqq \overline{\curl|_{[H^1_\theta(Y)]^3}}$ with the closure performed as an operation within $[L^2(Y)]^3$.
\end{proposition}
As the proof of this fact is analogous to the proof of Proposition \ref{t:guder} it is omitted.

The anticipated homogenisation theorem we deduce as a consequence of following our general abstract procedure reads as follows:
\begin{theorem}\label{t:maxhom} Let $\epsilon,\mu\in\mathcal{M}_{1,3}^\#$. Then, there exists $\kappa>0$ such that for all $\eta>0$ we have
\begin{multline*}
  \left\| \left(\begin{pmatrix}
       \epsilon\big( \tfrac{\cdot}{\eta} \big) & 0 \\ 0 & \mu\big( \tfrac{\cdot}{\eta} \big)
    \end{pmatrix} + \begin{pmatrix} 0 & -\curl \\ \curl & 0 \end{pmatrix}\right)^{-1} \right.- 
    \\  \mathcal{U}_\eta^{-1}\int_{[-\pi,\pi)^3}^\oplus \left. \left(\begin{pmatrix}
      \pi_{n(\theta)} \epsilon \pi_{n(\theta)} & 0 \\ 0 & \pi_{n(\theta)}\mu \pi_{n(\theta)}
    \end{pmatrix} + \frac{1}{\eta}\begin{pmatrix} 0 & -\curl_\theta \\  \curl_\theta & 0 \end{pmatrix}\right)^{-1}d\theta \mathcal{U}_\eta\right\|
     \leq \kappa \eta,
\end{multline*}
where 
\[
   \textnormal{n}(\theta) = e^{\i \langle\theta,\cdot\rangle_{\C^3}}\mathbb{C}^3 \oplus \{ \grad_\theta p; p\in H_\theta^1(Y)\bot e^{\i\langle \theta,\cdot\rangle_{\C^3}}\}.
\]
\end{theorem}
Unlike in the case of second-order elliptic systems with rapidly oscillating coefficients presented in Section \ref{s:example}, in general the object $\pi_{n(\theta)} a \pi_{n(\theta)}$, $a \in \mathcal{M}^{\#}_{1,3}$, cannot be expressed as the fibre-homogenised matrix given in Section \ref{sec:ahom}. Such a comparison in the Maxwell setting only occurs for a particular choice of right-hand side. Namely, the following result holds.
\begin{theorem}\label{t:ehom} Let $\eta>0$, $\epsilon,\mu\in\mathcal{M}^{\#}_{1,3}$, $F\in \kar(\dive)$. Then
  \begin{multline*}\mathcal{U}_\eta^{-1}\int_{[-\pi,\pi)^3}^\oplus  \left(\begin{pmatrix}
      \pi_{\textnormal{n}(\theta)} \epsilon \pi_{\textnormal{n}(\theta)} & 0 \\ 0 & \pi_{\textnormal{n}(\theta)}\mu \pi_{\textnormal{n}(\theta)}
    \end{pmatrix} + \frac{1}{\eta}\begin{pmatrix} 0 & -\curl_\theta \\  \curl_\theta & 0 \end{pmatrix}\right)^{-1}d\theta \mathcal{U}_\eta \begin{pmatrix} F \\ 0
    \end{pmatrix} \\
    =\mathcal{U}_\eta^{-1}\int_{[-\pi,\pi)^3}^\oplus  \left(\begin{pmatrix}
       \epsilon^{\textnormal{hom}}(\theta)\pi_{\textnormal{n}(\theta)}  & 0 \\ 0 & \mu^{\textnormal{hom}}(\theta)\pi_{\textnormal{n}(\theta)}
    \end{pmatrix} + \frac{1}{\eta}\begin{pmatrix} 0 & -\curl_\theta \\  \curl_\theta & 0 \end{pmatrix}\right)^{-1}d\theta \mathcal{U}_\eta \begin{pmatrix} F \\ 0
    \end{pmatrix}.
    \end{multline*}
\end{theorem}
In particular, from these two results, and the fact that $\widetilde{M}(\theta) = \begin{pmatrix}
\epsilon^{\rm hom}(\theta) & 0 \\ 0 & \mu^{\rm hom}(\theta)
\end{pmatrix}$ satisfies the assumptions of Theorem \ref{t:hom2}, we deduce the following result.
\begin{corollary}
Assume $\epsilon,\mu\in\mathcal{M}_{1,3}^\#$ and $F\in \kar(\dive)$. Then, there exists $\kappa>0$ such that for all $\eta>0$ we have
\begin{multline*}
\left\| \left(\begin{pmatrix}
\epsilon\big( \tfrac{\cdot}{\eta} \big) & 0 \\ 0 & \mu\big( \tfrac{\cdot}{\eta} \big)
\end{pmatrix} + \begin{pmatrix} 0 & -\curl \\ \curl & 0 \end{pmatrix}\right)^{-1} \begin{pmatrix} F \\ 0
\end{pmatrix} \right.- 
\\  \mathcal{U}_\eta^{-1}\int_{[-\pi,\pi)^3}^\oplus \left. \left(\begin{pmatrix}
\epsilon^{\rm hom}(\theta) & 0 \\ 0 & \mu^{\rm hom}(\theta)
\end{pmatrix} + \frac{1}{\eta}\begin{pmatrix} 0 & -\curl_\theta \\  \curl_\theta & 0 \end{pmatrix}\right)^{-1}d\theta \mathcal{U}_\eta \begin{pmatrix} F \\ 0
\end{pmatrix}\right\|
 \leq \kappa \eta \Vert F \Vert.
\end{multline*}
\end{corollary}

\begin{remark}
In fact, Theorem \ref{t:ehom} holds if $F \in [L^2(\mathbb{R}^3)]^3$ is such that, for each $\theta \in \Theta$, $\mathcal{U}_\eta F(\theta, \cdot)$ is an element of $\{ \grad_\theta p; p\in H_\theta^1(Y)\bot e^{\i\langle \theta,\cdot\rangle_{\C^3}}\}^\perp$. This is a consequence of Lemma \ref{l:ep} below. It is clear that $\ker(\dive)$ is a strict subset of such fields. 
\end{remark}
The next few paragraphs focus on a proof of Theorem \ref{t:maxhom}.  For this we will be applying Theorem \ref{t:mtgrtheta}, to the following setting
\begin{equation}\label{set:max}
\begin{split}    &\hspace{0\textwidth}H= H_\theta  = [L^2(Y)]^6,  \quad   N(\theta)  = [\textnormal{n}(\theta)]^2 = [ e^{\i \langle\theta,\cdot\rangle_{\C^3}}\mathbb{C}^3 \oplus \{ \grad_\theta p; p\in H_\theta^1(Y)\bot e^{\i\langle \theta,\cdot\rangle_{\C^3}}\}]^2, \\
 &\hspace{1.5cm} \begin{aligned}
 A(\theta) & = \begin{pmatrix}
                      0 & -\curl_\theta
                      \\ \curl_\theta & 0
                     \end{pmatrix}, & M(\theta) & = \begin{pmatrix}
                      \epsilon & 0
                      \\ 0 & \mu
                     \end{pmatrix}, &  \Theta &\coloneqq [-\pi,\pi)^3.
 \end{aligned} \end{split}
\end{equation}
Before we prove that Hypothesis \ref{hyp:grtheta} holds in this setting, let us study more closely the operator $\curl_\theta$. For this we introduce the following transformation.
\begin{definition}\label{def.vtheta} Let $\theta\in \Theta$. We define
\[
   \mathcal{V}_\theta \colon L^2(Y)\to \ell^2(\mathbb{Z}^3), u\mapsto \left(\langle u,e^{\i \langle \theta+2\pi z,\cdot\rangle_{\C^3}}\rangle\right)_{z\in \mathbb{Z}^3}.
\]
\end{definition}
Note that $\mathcal{V}_\theta$ is unitary with
\[
   \mathcal{V}_\theta^{-1} \big(c^{(z)}\big)_{z\in\mathbb{Z}^3} = \sum_{z\in\mathbb{Z}^3} c^{(z)} e^{\i \langle \theta+2\pi z,\cdot\rangle_{\C^3}},
\]
and the sum being convergent in $L^2(Y)$. 

In the following, we will employ the slight abuse of notation and reuse  $\mathcal{V}_\theta$ to denote the corresponding unitary operator from $[L^2(Y)]^3$ to $[\ell^2(\mathbb{Z}^3)]^3$, which acts component-wise as in the previous definition.
\begin{lemma}\label{l:foser} Let $\theta\in \Theta$. Then 
\[
\mathcal{V}_\theta    \curl_\theta \mathcal{V}_\theta^{-1} \big(c^{(z)}\big)_z = \left(\i(\theta+2\pi z)\times c^{(z)}\right)_z .
\]
In particular, $\curl_\theta=\curl_\theta^*$.
\end{lemma}
\begin{proof}
 The unitary equivalence follows by direct computation. The fact $\curl_\theta$ is selfadjoint now  follows from the fact that  the multiplication operator \[\big(c^{(z)}\big)_z \mapsto \left((\theta+2\pi z)\times c^{(z)}\right)_z\] is skew-selfadjoint.
\end{proof}
We now gather several relevant auxiliary results. For  $\theta\in \Theta$, recall 
$$
\textnormal{n}(\theta) = e^{\i \langle\theta,\cdot\rangle_{\C^3}}\mathbb{C}^3 \oplus \{ \grad_\theta p; p\in H_\theta^1(Y)\bot e^{\i\langle \theta,\cdot\rangle_{\C^3}}\},
$$ and set 
\begin{equation*}
   \textnormal{n}_1(\theta)\coloneqq e^{\i \langle\theta,\cdot\rangle_{\C^3}}\mathbb{C}^3,
  \quad \textnormal{r}(\theta)\coloneqq \textnormal{n}(\theta)^{\bot_{[L^2(Y)]^3}}.
\end{equation*}
As usual, let  $\iota_{\textnormal{n}(\theta)}$, $\iota_{\textnormal{r}(\theta)}$ be the canonical embeddings and  $\pi_{\textnormal{n}(\theta)}$, $\pi_{\textnormal{r}(\theta)}$ the orthogonal projections.
\begin{lemma}\label{l:co} Let $\theta\in \Theta$, $u=\sum_{z\in\mathbb{Z}^3} c^{(z)}e^{\i\langle \theta+2\pi z,\cdot\rangle_{\C^3}} \in [L^2(Y)]^3$. Then
\[
   u\in \textnormal{r}(\theta) \iff c^{(0)}=0\quad \& \quad \forall z\in \mathbb{Z}^3\colon c^{(z)} \bot (\theta+2\pi z). 
\] 
\end{lemma}
\begin{proof}
 Let $u\in \textnormal{r}(\theta)$. Then, clearly $c^{(0)}=0$. Moreover, we compute for $z\in \mathbb{Z}^3\setminus\{0\}$:
 \begin{align*}
   \langle c^{(z)},-\i(\theta+2\pi z)\rangle & = \sum_{j=1}^3 \langle c_j^{(z)},-\i(\theta+2\pi z)_j\rangle
   \\ & = \sum_{j=1}^3 \langle \langle u_j,e^{\i\langle \theta+2\pi z,\cdot\rangle_{\C^3}}\rangle,-\i(\theta+2\pi z)_j\rangle
   \\ & = \langle u,\i(\theta+2\pi z)e^{\i\langle \theta+2\pi z,\cdot\rangle_{\C^3}}\rangle
   \\ & = \langle u,\grad_\theta e^{\i\langle \theta+2\pi z,\cdot\rangle_{\C^3}}\rangle
   \\ & = 0.
 \end{align*}
On the other hand, assume that the Fourier coefficients satisfy the properties mentioned on the right-hand side of the claimed equivalence. Then $c^{(0)}=0$ implies $\langle u,e^{\i\langle \theta,\cdot\rangle_{\C^3}}\gamma \rangle=0$ for all $\gamma \in \mathbb{C}^3$. Next, let $p\in H_\theta^1(Y)\bot e^{\i\langle\theta,\cdot\rangle_{\C^3}}\mathbb{C}^3$. From the identity $p=\sum_{z\in\mathbb{Z}^3\setminus\{0\}} c^{(z)}_pe^{\i\langle \theta+2\pi z,\cdot\rangle_{\C^3}}\in H_\theta^1(Y)\bot e^{\i\langle\theta,\cdot\rangle_{\C^3}}$ we deduce that
\begin{align*}
   \langle u,\grad_\theta p\rangle_{[L^2(Y)]^3} & = \langle u, \grad_\theta \sum_{z\in\mathbb{Z}^3\setminus\{0\}} c_p^{(z)}e^{\i\langle \theta+2\pi z,\cdot\rangle_{\C^3}}\rangle
   \\ &  = \langle u, \sum_{z\in\mathbb{Z}^3\setminus\{0\}} \i(\theta+2\pi z)c_p^{(z)}e^{\i\langle \theta+2\pi z,\cdot\rangle_{\C^3}}\rangle
   \\ & = \sum_{z\in\mathbb{Z}^3\setminus\{0\}} \langle c^{(z)},\i(\theta+2\pi z)\rangle \overline{c_p^{(z)}}
   \\ & = 0,
\end{align*}
which establishes the claim.
\end{proof}

\begin{proposition}\label{p:curln} Let $\theta\in \Theta$. Then 
\[
   \curl_\theta \pi_{\textnormal{n}(\theta)}=\i\theta\times \pi_{\textnormal{n}_1(\theta)}.
\] Moreover, we have
\[
   \pi_{\textnormal{n}(\theta)}\curl_\theta \subseteq\curl_\theta \pi_{\textnormal{n}(\theta)}.
\]
\end{proposition}
\begin{proof}
Let $u\in [L^2(Y)]^3$.
Then
\[
   u = \pi_{\textnormal{n}(\theta)} u + \pi_{\textnormal{r}(\theta)} u = e^{\i \langle\theta,\cdot\rangle_{\C^3}} \gamma +\grad_\theta p +\pi_{\textnormal{r}(\theta)} u,
\]
for some $\gamma\in\mathbb{C}^3$ and $p\in H^1_\theta(Y)\bot e^{\i\langle\theta,\cdot\rangle_{\C^3}}$. Since $\rge(\grad_\theta)\subseteq \kar(\curl_\theta)$, we obtain that $\pi_{\textnormal{n}(\theta)}u\in \dom(\curl_\theta)$. Moreover, we compute
\[
   \curl_\theta\pi_{\textnormal{n}(\theta)} u =\curl_\theta\left( e^{\i \langle\theta,\cdot\rangle_{\C^3}} \gamma +\grad_\theta p\right)=\i\theta\times e^{\i \langle\theta,\cdot\rangle_{\C^3}} \gamma = \i\theta\times \pi_{\textnormal{n}_1(\theta)}u.
\]
This shows the first desired assertion. 

Next, assume in addition that $u\in \dom(\curl_\theta)$. Then, $\pi_{\textnormal{r}(\theta)}u\in \dom(\curl_\theta)$ and 
\[
   \curl_\theta u = \i\theta\times \pi_{\textnormal{n}_1(\theta)}u+ \curl_\theta\pi_{\textnormal{r}(\theta)} u =\i\theta\times e^{\i \langle\theta,\cdot\rangle_{\C^3}} \gamma +\curl_\theta\pi_{\textnormal{r}(\theta)} u.
\]
From $\i\theta\times e^{\i \langle\theta,\cdot\rangle_{\C^3}} \gamma\in \textnormal{n}_1(\theta)\subseteq \textnormal{n}(\theta)$, we infer that
\[
   \pi_{\textnormal{r}(\theta)}\curl_\theta u
    = \pi_{\textnormal{r}(\theta)}\curl_\theta\pi_{\textnormal{r}(\theta)} u.
\]
Lemma \ref{l:foser} implies that there exists some $(c^{(z)})_{z\in\mathbb{Z}^3}$ in $\mathbb{C}^3$  such that
\[
   \curl_\theta\pi_{\textnormal{r}(\theta)} u = \sum_{z\in \mathbb{Z}^3\setminus\{0\}} \i(\theta+2\pi z)\times c^{(z)}e^{\i\langle\i(\theta+2\pi z),\cdot\rangle_{\C^3}}.
\]
Furthermore, since $\langle (\theta+2\pi z),\i(\theta+2\pi z)\times c^{(z)}\rangle=0$, it follows from Lemma \ref{l:co} that $\curl_\theta\pi_{\textnormal{r}(\theta)} u \in \textnormal{r}(\theta)$. Hence,
\[
 \pi_{\textnormal{r}(\theta)}\curl_\theta u = \pi_{\textnormal{r}(\theta)}\curl_\theta\pi_{\textnormal{r}(\theta)} u=\curl_\theta\pi_{\textnormal{r}(\theta)} u,
\]
which together with Lemma \ref{l:Athetainv}, for $H = [L^2(Y)]^3$, $A = \curl_\theta$ and $U = \textnormal{r}(\theta)$, yields the second desired assertion.
\end{proof}
\begin{proposition}\label{p:poin} Let $\theta\in\Theta$, $u\in \dom(\curl_\theta)\cap\textnormal{r}(\theta)$. Then 
\[
   \|u\|\leq \tfrac{1}{\pi}\|\curl_{\theta} u\|.
\] 
\end{proposition}
\begin{proof}
 There exists $(c^{(z)})_{z\in\mathbb{Z}^3}$ in $\mathbb{C}^3$ with $u=\sum_{z\in\mathbb{Z}^3} c^{(z)} e^{\i\langle \theta+2\pi z,\cdot\rangle_{\C^3}}$.
 By Lemma \ref{l:co}, we have that $c^{(0)}=0$ and $c^{(z)}\bot \theta+2\pi z$ for all $z\in\mathbb{Z}^3$. In particular, we get
 \[
    \|(\theta+2\pi z)\times c^{(z)}\|=\|(\theta+2\pi z)\|\|c^{(z)}\|\geq \pi \|c^{(z)}\|\quad(z\in\mathbb{Z}^3\setminus\{0\}).
 \]
 Thus, by Lemma \ref{l:foser},
 \begin{align*}
  \|\curl_{\theta} u\|^2 &=\Big\|\sum_{z\in\mathbb{Z}^3\setminus\{0\}} \i(\theta+2\pi z)\times c^{(z)} e^{\i\langle \theta+2\pi z,\cdot\rangle_{\C^3}}\Big\|^2
   \\ & = \sum_{z\in\mathbb{Z}^3\setminus\{0\}} \left\|\i(\theta+2\pi z)\times c^{(z)} e^{\i\langle \theta+2\pi z,\cdot\rangle_{\C^3}}\right\|^2
   \\ & = \sum_{z\in\mathbb{Z}^3\setminus\{0\}} \left\|(\theta+2\pi z)\times c^{(z)}\right\|^2
   \\ &\geq \pi^2 \sum_{z\in\mathbb{Z}^3\setminus\{0\}} \left\| c^{(z)}\right\|^2 = \pi^2 \|u\|^2.\qedhere
 \end{align*}
\end{proof}
In the setting  \eqref{set:max}, with the Propositions \ref{p:clo} \ref{clo2}, \ref{p:curln} and \ref{p:poin}, we can readily demonstrate Hypothesis \ref{hyp:grtheta} \ref{en:grt1}-\ref{en:grt3} for the setting \eqref{set:max}. In particular, the assumptions of Theorem \ref{t:mtgr} hold. To argue as in the proof of Theorem \ref{t:quanthom} and obtain a proof for Theorem \ref{t:maxhom}, it remains to prove Hypothesis \ref{en:grt5}.

\begin{proposition}
	\label{p:wmmax1} Let $(\theta_k)_{k\in \mathbb{N}}$ be a convergent sequence in $\Theta$, $\theta\coloneqq \lim_{k\to\infty}\theta_k$. Let $(u_k)_{k\in\mathbb{N}}$ in $\dom(\curl_{\theta_k}) \subseteq [L^2(Y)]^3$ weakly converge to some limit $u$. Assume that $(\curl_{\theta_k} u_k)_{k\in\mathbb{N}}$ is bounded. Then $u\in\dom(\curl_\theta)$ and 
		\[ 
		\curl_{\theta_k}u_k\rightharpoonup \curl_\theta u.
		\]
	
\end{proposition}
\begin{proof}
Recalling Definition \ref{def.vtheta}, we have 
\[
\begin{aligned}
u_k = \mathcal{V}_{\theta_k}^{-1} \big(c^{(z)}_k\big)_{z \in \Z^3} = \sum_{z \in \Z^3} c^{(z)}_k e^{\i\langle \theta_k + 2\pi z, \cdot \rangle_{\C^3}}, & \quad & u = \mathcal{V}_\theta^{-1} \big(c^{(z)}\big)_{z \in \Z^3} = \sum_{z \in \Z^3} c^{(z)} e^{\i\langle \theta + 2\pi z, \cdot \rangle_{\C^3}},
\end{aligned}
\]
for  $\big(c^{(z)}_k\big)_{z \in \Z^3} = \left(\langle u_k,e^{\i \langle \theta_k+2\pi z,\cdot\rangle_{\C^3}}\rangle\right)_{z\in \mathbb{Z}^3}, \big(c^{(z)}\big)_{z \in \Z^3} = \left(\langle u,e^{\i \langle \theta+2\pi z,\cdot\rangle_{\C^3}}\rangle\right)_{z\in \mathbb{Z}^3} \in \ \ell^2(\mathbb{Z}^3)$.
Lemma \ref{l:foser} states that
\[
\mathcal{V}_{\theta_k} \curl_{\theta_k}u_k  = \sum_{z \in \Z^3} \i(\theta_k + 2\pi z) \times c^{(z)}_k e^{\i\langle \theta_k + 2\pi z, \cdot \rangle_{\C^3}}.
\] 
Passing to the point-wise limit above, we determine that 
\[
\mathcal{V}_{\theta_k} \curl_{\theta_k}u_k \rightarrow  \sum_{z \in \Z^3} \i(\theta + 2\pi z) \times c^{(z)} e^{\i\langle \theta + 2\pi z, \cdot \rangle_{\C^3}},
\]
as $k \rightarrow \infty$. This, plus the assumption that $( \curl_{\theta_k} u_k )_{k \in \N}$ is bounded, implies the desired assertion. 
\end{proof}
\begin{proposition}
	\label{p:wmmax2} 
Let  $\ep, \mu \in \mathcal{M}^{\#}_{1,3}$, and $\eta >0$. Assume setting \eqref{set:max}.	Let $T : \Theta \rightarrow L([L^2(Y)]^6)$ be given by 
	\[
	\theta \mapsto  \left( \begin{pmatrix} \ep  & 0 \\ 0& \mu \end{pmatrix} + \frac{1}{\eta}\begin{pmatrix} 0 & - \curl_\theta \\ \curl_\theta & 0\end{pmatrix}\right)^{-1}.
	\]
	Then, $T$ is weakly continuous.
\end{proposition}
\begin{proof}
Let $( \theta_k )_{k \in \N}$ be convergent in $\Theta$ to some limit $\theta$. We need to prove that for $f \in [L^2(Y)]^6$, then  the sequence 
$( T(\theta_k) f)_{k \in \N} $ is weakly convergent in $[L^2(Y)]^6$ to the limit $T(\theta) f$.

By Lemma \ref{l:invblty}, 
\[
\sup_{\theta \in \Theta}\Vert T(\theta) \Vert \le \tfrac{1}{\nu},
\]
where $\nu>0$ is such that $\Re \ep \ge \nu 1_{\C^3}$, $ \Re \mu \ge \nu 1_{\C^3}$. Therefore, $( T(\theta_k) f )_{k \in \N} $ weakly converges, up to a subsequence, to  some $u$. Moreover,
\[
\sup_{k \in \N} \Big\Vert \begin{pmatrix} 0 & - \curl_\theta \\ \curl_\theta & 0\end{pmatrix}T(\theta_k) f  \Big\Vert < \infty,
\]
and therefore, by Proposition \ref{p:wmmax1}, we deduce that $u = T(\theta) f$. Since $u$ is unique, the whole sequence $( T(\theta_k) f )_{k \in \N} $ weakly converges and the proof is established.
\end{proof}
The proof of Hypothesis \ref{hyp:grtheta} \ref{en:grt5} now follows from Proposition \ref{p:wmmax2}, as $T$ is weakly continuous and therefore weakly measurable. 
  
We conclude this section by proving Theorem \ref{t:ehom}. This result is a consequence of the following proposition and the assertion $\mathcal{U}_\eta \dive=\int_\Theta^\oplus \dive_\theta d\theta U_\eta$:
\begin{proposition}\label{p:ehom} Let $\eta>0$, $\eps,\mu\in\mathcal{M}_{1,3}^\#$, $f\in\kar(\dive_\theta)$. Then
 \begin{multline*}
  \left(\begin{pmatrix}
      \pi_{\textnormal{n}(\theta)} \epsilon \pi_{\textnormal{n}(\theta)} & 0 \\ 0 & \pi_{\textnormal{n}(\theta)}\mu \pi_{\textnormal{n}(\theta)}
    \end{pmatrix} + \frac{1}{\eta}\begin{pmatrix} 0 & -\curl_\theta \\  \curl_\theta & 0 \end{pmatrix}\right)^{-1} \begin{pmatrix} f \\ 0
    \end{pmatrix}\\
    =\left(\begin{pmatrix}
       \epsilon^{\textnormal{hom}}(\theta)\pi_{\textnormal{n}(\theta)}  & 0 \\ 0 & \mu^{\textnormal{hom}}(\theta)\pi_{\textnormal{n}(\theta)}
    \end{pmatrix} + \frac{1}{\eta}\begin{pmatrix} 0 & -\curl_\theta \\  \curl_\theta & 0 \end{pmatrix}\right)^{-1}\begin{pmatrix} f \\ 0
    \end{pmatrix}.
 \end{multline*}
\end{proposition}
For the proof of this proposition, we will utilise the following result. 
\begin{proposition}\label{l:ep} Let $\theta\in \Theta$, $\epsilon\in\mathcal{M}^\#_{1,3}$, $E\in \textnormal{n}(\theta)$ and $f\in \textnormal{n}_1(\theta)\oplus \textnormal{r}(\theta)$. Then
\[
    \pi_{\textnormal{n}(\theta)}\epsilon E = f \iff \epsilon^{\textnormal{hom}}(\theta) E=f.
\]
\end{proposition}
\begin{proof}
  Let $\textnormal{n}_2(\theta)\coloneqq \textnormal{n}(\theta)\ominus\textnormal{n}_1(\theta)=\{ \grad_\theta p; p\in H^1_\theta(Y)\bot e^{\i\langle\theta,\cdot\rangle_{\C^3}}\}$. Assume that $\pi_{\textnormal{n}(\theta)}\epsilon E = f$. Then, for all $\phi\in \textnormal{n}_2(\theta)$ we deduce that
  \[
     \langle \pi_{\textnormal{n}(\theta)}\epsilon E,\phi\rangle = \langle f,\phi\rangle = 0.
  \]
  Hence, as $\pi_{\textnormal{n}(\theta)}\phi=\phi$ for all $\phi\in\textnormal{n}_2(\theta)$ we obtain
  \[
      \langle \epsilon E,\phi\rangle = 0 \quad(\phi\in \textnormal{n}_2(\theta)).
  \]
  Moreover, projecting on $\textnormal{r}(\theta)$ reveals that
  \[
     \pi_{\textnormal{r}(\theta)} f = \pi_{\textnormal{r}(\theta)} \pi_{\textnormal{n}(\theta)} \epsilon E = 0.
  \]
  Thus, $f=\pi_{\textnormal{n}_1(\theta)}f$ and so
  \[
     \langle \epsilon  E,\phi\rangle = \langle f,\phi\rangle \quad (\phi\in\textnormal{n}(\theta)),
  \]
  which readily gives  
  \[
     \epsilon^{\textnormal{hom}}(\theta)E=f.  \]
  The other implication is similar.
  \end{proof}

\begin{proof}[Proof of Proposition \ref{p:ehom}]
 Let $(E,H)\in [\dom(\curl_\theta)]^2$ be such that
 \begin{equation*}
  \left(\begin{pmatrix}
      \pi_{\textnormal{n}(\theta)} \epsilon \pi_{\textnormal{n}(\theta)} & 0 \\ 0 & \pi_{\textnormal{n}(\theta)}\mu \pi_{\textnormal{n}(\theta)}
    \end{pmatrix} + \frac{1}{\eta}\begin{pmatrix} 0 & -\curl_\theta \\  \curl_\theta & 0 \end{pmatrix}\right) \begin{pmatrix} E \\ H
    \end{pmatrix}=\begin{pmatrix} f \\ 0
    \end{pmatrix}.
 \end{equation*}
 Equivalently,  \begin{align*}
   \pi_{\textnormal{n}(\theta)} \epsilon \pi_{\textnormal{n}(\theta)} E - \tfrac{1}{\eta}\curl_\theta H  = f, & \quad &  \pi_{\textnormal{n}(\theta)} \mu \pi_{\textnormal{n}(\theta)} H + \tfrac{1}{\eta}\curl_\theta E  = 0.
 \end{align*}
Let us focus on the first equation. One implication of this equations is that $\curl_\theta H\in \textnormal{n}_1(\theta)\oplus \textnormal{r}(\theta)$. Thus, for $\tilde f\coloneqq f+\tfrac{1}{\eta}\curl_\theta H$ and $\tilde E\coloneqq \pi_{\textnormal{n}(\theta)} E$, we have
 \[
   \pi_{\textnormal{n}(\theta)}\epsilon \tilde E = \tilde f\in \textnormal{n}_1(\theta)\oplus \textnormal{r}(\theta),
 \]
and Proposition \ref{l:ep} implies
 \[
   \epsilon^{\textnormal{hom}}(\theta)\tilde{E} = \tilde{f}.
 \]
 The argument for the second equation is completely analogous.  Thus, the desired assertion holds.
\end{proof}

The proof of Theorem \ref{t:ehom} now follows by applying Proposition \ref{p:ehom}  pointwise for any $\theta\in \Theta$.

\section*{Acknowledgements}
S. Cooper was supported by the EPSRC grant EP/M017281/1 (``Operator asymptotics, a new approach to length-scale interactions in metamaterials"). M. Waurick is grateful for the financial support of the EPSRC grant EP/L018802/1 (``Mathematical foundations of metamaterials: homogenisation, dissipation and operator theory'').


\begin{thebibliography}{9}
\bibitem{BeLiPa}{Bensoussan, A., Lions, J.-L., and Papanicolaou, G. C., 1978. {\it Asymptotic Analysis for Periodic Structures}, Amsterdam: North-Holland.
	} 
\bibitem{BiSu}{Birman, M. Sh., and Suslina, T. A., 2004. Second order periodic differential operators. Threshold
	properties and homogenisation. {\it St. Petersburg. Math. J.} 15(5), pp. 639-714.
}
\bibitem{ChCo}{ Cherednichenko, K.D., Cooper, S., 2016. Resolvent Estimates for High-Contrast Elliptic Problems with Periodic Coefficients. {\it Arch Rational Mech Anal} 219, pp. 1061-1086. 
}
\bibitem{ChWa}{Cherednichenko, K., Waurick, M., 2017. Resolvent estimates in homogenisation of periodic
	problems of fractional elasticity. {\it Preprint, University of Bath.}
}
\bibitem{CoVa}{Conca, C., Vanninathan, M., 1997.  Homogenisation of periodic structures via Bloch decomposition. {\it	SIAM J. Appl. Math.} 57, pp. 1639–1659.
}
\bibitem{tEGoWa}{ter Elst, A.F.M., Gordon, G., Waurick, M. 2017. The Dirichlet-to-Neumann operator for divergence form problems. {\it Preprint, University of Auckland}
}
\bibitem{Gr}{Griso, G., 2006. Interior error estimate for periodic homogenisation. {\it Anal. Appl.} 4(1), pp. 61–79.
}
\bibitem{KeLiSh}{Kenig, C.E., Lin, F., Shen, Z., 2012.  Convergence rates in L2 for elliptic homogenization
problems. {\it Arch. Ration. Mech. Anal.} 203(3), pp. 1009–1036.
}
\bibitem{Pi}{ Picard, R., McGhee, D., 2011. {\it Partial Differential Equations: A unified Hilbert Space Approach}, Expositions in Mathematics. DeGruyter, Berlin, 55.
}
\bibitem{PiTrWaWe}{ Picard, R.,  Trostorff, S.,  Waurick, M., and Wehowski. M., 2013. On Non-Autonomous Evolutionary Problems. {\it Journal of Evolution Equations}, 13, pp. 751-776.
}
\bibitem{Se}{ Senik, N. N., 2017. Homogenization for non-self-adjoint locally periodic elliptic operators. {\it arXiv}: 1703.02023.
}
\bibitem{Su1}{Suslina, T. A., 2013. Homogenization of the Dirichlet problem for elliptic systems: L2-
	operator error estimates. {\it Mathematika} 59, pp. 463–476.
}
\bibitem{Su2}{Suslina, T. A.,  2012. Homogenization of the Neumann Problem for Elliptic Systems with Periodic Coefficients. {\it SIAM Journal on Mathematical Analysis} 45(6).
}
\bibitem{Wa1}{Waurick M., 2013. Homogenization of a class of linear partial differential equations. 
	{\it Asymptotic Analysis} 82, pp. 271-294.
}
\bibitem{Wa2}{ Waurick, M., 2014. 	Homogenization in fractional elasticity. {\it SIAM J. Math. Anal.} 46(2), pp. 1551-1576.	
}
\bibitem{Wa3}{Waurick, M., 2016. On the homogenization of partial integro-differential-algebraic equations. {\it Operators and Matrices}, 10(2), pp. 247-283.
}
\bibitem{J1}{Zhikov, V. V., 1989. Spectral approach to asymptotic problems in diffusion. {\it Differ. Equ.} 25, pp. 33–39.
}
\bibitem{JKO}{Jikov, V.V., Kozlov, S.M., Oleinik, O.A., 1994. {\it Homogenization of Differential Operators
		and Integral Functionals,} Springer, Berlin.
}
\bibitem{JPa}{Zhikov, V.V., Pastukhova, S.E., 2005. On operator estimates for some problems in homogenization
	theory. {\it Russ. J. Math. Phys.} 12(4),  pp. 515–524.
}
\end{thebibliography}
\end{document}